\documentclass[english]{article}
\usepackage[a4paper, total={6in, 9in}]{geometry}

\usepackage{authblk}
\usepackage[T1]{fontenc}
\usepackage[latin9]{inputenc}
\usepackage[english]{babel}
\usepackage{graphicx}
\usepackage{float}
\usepackage{bbm}
\usepackage{graphicx}%
\usepackage{amsmath,amssymb,amsfonts,amsthm}%
\usepackage{mathrsfs}%
\usepackage[title]{appendix}%
\usepackage{xcolor}%
\usepackage{rotating}

\newtheorem{theorem}{Theorem}
\newtheorem{corollary}{Corollary}

\newtheorem{conjecture}{Conjecture}
\newtheorem{rem}{Remark}
\newtheorem{claim}{Claim}
\newcommand{\E}{\mathrm{E}}
\newcommand{\Var}{\mathrm{Var}}

\usepackage{natbib}
\setcitestyle{authoryear,open={(},close={)}}

\newcommand{\Conv}{%
  \mathop{\scalebox{1.5}{\raisebox{-0.2ex}{$\circledast$}}
  }
}

\begin{document}
\title{Effect of sampling on the descriptive distributions of correlated random walks}
\author[1]{Joseph D. Bailey}
\author[1]{Jessica Claridge}
\affil[1]{School of Mathematics, Statistics and Actuarial Science, University of Essex, UK}

\maketitle

\begin{abstract}
Random walks are commonly used to model movement throughout the sciences, from the motion of particles and molecules to the observed behaviour of animals and crowds.  The correlated random walk, which assumes a level of persistence between movement directions, has become ubiquitous in the analysis and modelling of movement in recent times.  Whilst many properties of the correlated random walk are known, there are still many which are not fully understood and, therefore, under utilised in movement data analysis.  Here we consider the effect that sub-sampling has on the descriptive distributions of correlated random walks.  Our work demonstrates the connection between the distributions of turning angles and step-lengths that characterise a correlated random walk, along with the resulting distributions found after sub-sampling.  We provide examples for where this approach could aid in movement analysis as well as determining future ways the work could be extended.
\end{abstract}

\section{Introduction}
Random walks (RW) appear throughout applied mathematics and the sciences; from the motion of particles in physical sciences \citep{patlak1953random,alt1980biased} to the modelling of animal movement in life sciences \citep{berg1993random,codling2008random,bailey2018navigational}, and from predicting financial markets in economics \citep{fama1995random} to the programming of robotics in computer science \citep{dimidov2016random}.

Early work concerning RW models assumed turns were uniformly distributed on the unit circle.  These simple random walks (SRW) give movement akin to simple Brownian motion and, though simplistic, they have been successfully incorporated into many natural and physical processes including in the modelling of gas particles \citep{weiss1983random}, solid state physics \citep{weiss1983brandom}, epidemiology \citep{fofana2017mechanistic,white2018dynamic} and ecological phenomena such as dispersal \citep{skellam1951random,turchin1998quantitative,okubo2001diffusion}, foraging behaviour \citep{bell1990} and predator-prey behaviour \citep{kareiva1990population}.  Extensions to the SRW model aim to more closely mimic real-world drivers of movement, such as including attraction/repulsion towards a global direction or location, due to some external or internal  force, for example a gradient function from a field force or a reaction to olfactory sense when migrating or navigating towards a food source.  This bias in the direction of movement leads to the biased random walk (BRW) model, which has been used in the analysis of many processes including chemotaxis \citep{alt1980biased,codling2008random}, diffusion in anomalous media \citep{bouchaud1990anomalous}, migration \citep{saul2012individual,papastamatiou2013telemetry}, foraging \citep{farnsworth1999grazers,barton2009evolution}, home-range behaviour \citep{borger2008there,breed2017predicting} and the movement of microorganisms \citep{hill1997biased}.  

An alternative approach is to consider how the direction of movement at each step changes relative to the direction of the previous step; a correlation in these headings results in the correlated random walk (CRW), which has become a popular null model in animal movement \citep{fagan2014correlated}.    
In general, movement is therefore described as a highly correlated, giving near straight line movement, or lowly correlated, in which movement appears more random and similar to a SRW.  

CRWs have a wide range of applications within ecology including the modelling of group behaviour and movement \citep{haydon2008socially,codling2014copycat}, modelling invasive spread of plants \citep{andersen2005risk} and are regularly used in the analysis of individual movement \citep{mcclintock2018momentuhmm,grant2018predicting,bailey2021walking}.  Outside of mathematical ecology, CRW have continued to be utilised and developed in fields including physics \citep{cressoni2013exact,sadjadi2015persistent,yuste2024gaseous}, soft matter physics \citep{liu2021correlated}, materials science \citep{bonilla2012multicomponent} and robotics, \citep{dimidov2016random,pang2021effect}. 

One natural problem within the RW framework is the question of sampling, which entails selecting a subset of the points in the random walk (i.e. every 10th point) and analysing the subsequent path found by connecting these.  An open problem concerns how the sampling rate effects the properties of the random walk.  In terms of data analysis, sampling from the observed data is used for a variety of reasons, such as reducing noise \citep{bailey2021micropersonality}, reducing volume of data \citep{gupte2022guide}, and allowing for meaningful/appropriate scales of movement e.g. recording the movement of large ungulates at multiple fixes per second will not give reliable results compared to taking values every few minutes \citep{nathan2022big,schlagel2016robustness,potts2018finding}.  However, the effects on the underlying model that this rescaling has is not fully understood, despite the need to ensure the appropriateness of any analyses on such datasets, regardless of scale, being well known \citep{schlagel2016robustness}.  As RW models assume that reorientation events occur precisely at the frequency of the recorded data \citep{codling2005sampling}, much previous work has concerned the identification of meaningful turning points (reorientation events) across the movement trajectory and the effect that sampling has on the potential to miss or mis-locate such points \citep{codling2005sampling,nathan2022big}.

Concentrating on pure CRW processes, \citet{codling2005sampling} considered the effect of summary statistics when a CRW is sampled, demonstrating via large-scale simulation that angular deviation observes a linear relation with the square root of the sampling rate, as well as generating expressions for the mean and variance of the distribution of speeds  (analogous to step-lengths); although, they note that this relationship breaks down as the RW becomes closer to a SRW.  \citet{rosser2013effect} extended on this work by further simulation to demonstrate the complex relationship between the sampling rate and the turning angle distributions.  They also determine an expression for the joint distribution between step lengths and turning angles when the number of true reorientation events within the sub-sampled RW is unknown, achieving an analytical expression that holds for large sampling rates and for certain constraints on the initial turning angle distribution.  \citet{nams2013sampling}, demonstrated through extensive simulations that by sub-sampling a CRW, autocorrelation between subsequent turning angles can artificially be introduced, which therefore causes the sampled RW to no longer be a strict CRW; although in practice movement data with some level of autocorrelation is still often treated and analysed as a CRW.  In a related but subtly different approach, \citet{schlagel2016robustness} focussed on identifying movement models that were `robust' under temporal sampling, allowing specific models to be fitted to data at varying temporal resolutions, where robust was defined by specific mathematical properties of the movement model's characteristic functions (see section 3.1 in \citet{schlagel2016robustness}).  However, their work was limited to models with uniform turning angle distribution only.

Here we extend on this previous work by determining the form of the probability density function (PDF) for the turning angles of the sub-sampled CRW.  Our method is not limited to any specific turning angle distributions (e.g. wrapped Cauchy, wrapped normal, von Mises, SWS distributions \citep{jammalamadaka2001topics}), rather it is valid for all symmetric circular distributions.  We provide a precise analytical expression for the turning angle and step-length distributions when a CRW is sampled at every other point.  We continue to provide a method that provides the distribution of the turning angels for any sampling rate of the form $2^k$, where $k$ is a positive integer, up to very close approximation. The approach used here gives a geometric interpretation of a cause for the autocorrelation observed by \citet{nams2013sampling}. Finally, we provide examples for where the approach of comparing the known and observed sub-sampled distributions of turning angles can be used in the analysis of movement data.

\section{Background}
Consider a 2-dimension discrete time, continuous space random walk (RW) in the plane, given by $\left(X\right)_{n}=\{X_{1},X_{2},\ldots,X_{n}\}$ where $X_{t}=(x_{t},y_{t})$ denotes the location of the RW at time step $t$.  Successive locations are given by $X_{t+1}=X_t+(\Delta x_t, \Delta y_t)$, where $\Delta x_t, \Delta y_t$ are the changes in the $x$ and $y$ coordinate respectively at time $t+1$, therefore giving $\Delta x_t= x_{t+1}-x_t$ and $\Delta y_t= y_{t+1}-y_t$.

Any such RW can be considered as a step-turn sequence, where the location of successive steps is given by a step-length, $\ell$, equivalent to the distance between successive points, and a turning angle, $\theta$, which denotes the angle between the previous direction of movement and the new direction.  

Where
\begin{equation}
\ell_{t}=\Vert\overrightarrow{X_{t}X_{t+1}}\Vert=\sqrt{(x_{t+1}-x_{t})^{2}+\left(y_{t+1}-y_{t}\right)^{2}}=\sqrt{\Delta x_t^2+\Delta y_t^2}
\end{equation}
and 
\begin{equation}
 \theta_{X_t}= \angle\left(\overrightarrow{X_{t-1}X_{t}},\overrightarrow{X_{t}X_{t+1}}\right) \Rightarrow \cos\theta_{X_t}=\frac{\overrightarrow{X_{t-1}X_{t}}\cdot\overrightarrow{X_{t}X_{t+1}}}{\Vert\overrightarrow{X_{t-1}X_{t}}\Vert\Vert\overrightarrow{X_{t}X_{t+1}}\Vert}
\end{equation}
for $1\leq t \leq n-1$, with $(x_0,y_0), (x_1,y_1)$ given a priori.  Therefore, in this way a RW can be described by the sequence $(\ell_t,\theta_{X_t})$ with step-lengths modelled by some distribution on the positive reals, $\ell_t\sim \Lambda$, and turning angles modelled by a distribution on the unit circle $\theta_t\sim f_\circ$, i.e. $f_\circ$ has domain $(-\pi,\pi]$. (See Appendix A for more information on these distributions.)

The simplest form of a discrete time RW is the simple random walk (SRW), which is closely related to Brownian motion.  In a SRW the changes in the $x$ and $y$ coordinates at each time point are independently and identically drawn (i.i.d) from the normal distribution, $\Delta x, \Delta y \sim N(0,\sigma^2)$.  This leads to the equivalent step-turn process having turning angles i.i.d from the uniform distribution, $\theta_{t}\sim \text{Unif}(-\pi,\pi)$, and step-lengths also being i.i.d, following the Rayleigh distribution, $\ell_{t}\sim R(l;\sigma^2)$, which has mean, $\E[\ell]=\frac{\sigma\sqrt{2\pi}}{2}$, and variance, $\Var[\ell]=2\sigma^2\left(1-\frac{\pi}{4}\right)$ \citep{petrovskii2014multiscale,Ahmed2021}.

In a correlated random walk (CRW), in which direction of movement is expected to persist between proceeding steps, the bearings between successive steps (angle between steps taken from due North, x-axis etc.) are correlated, with successive steps locally biased towards the same direction, however the turning-angles are not correlated.   A standard approach assumes the turning angles are described by a symmetric, uni-modal distribution with mean value 0, such as the wrapped Cauchy, wrapped normal or von Mises distributions (see Appendix A for more information on circular distributions).  The relative `straightness' of a CRW is governed by how peaked the turning angle distribution is around the value 0, which is described by the length of the mean resultant vector, $\rho$, calculated as the root of the sum of the squares of the distribution's first trigonometric moments (Appendix A).  Values close to $1$ give near straight-line movement, with values close to 0 giving more random motion (note $\rho=0$ is equivalent to the circular uniform distribution where all angles in $(-\pi,\pi]$ are equally likely) (Fig.~\ref{fig:CRW examp}).  It is also often assumed in a CRW that each step-length, $\ell_i$, and each turning angle, $\theta_{X_i}$, are i.i.d with step-lengths and turning-angles neither jointly- nor auto-correlated, which we shall assume  throughout.

\begin{figure}
    \centering
    \includegraphics[width=4in]{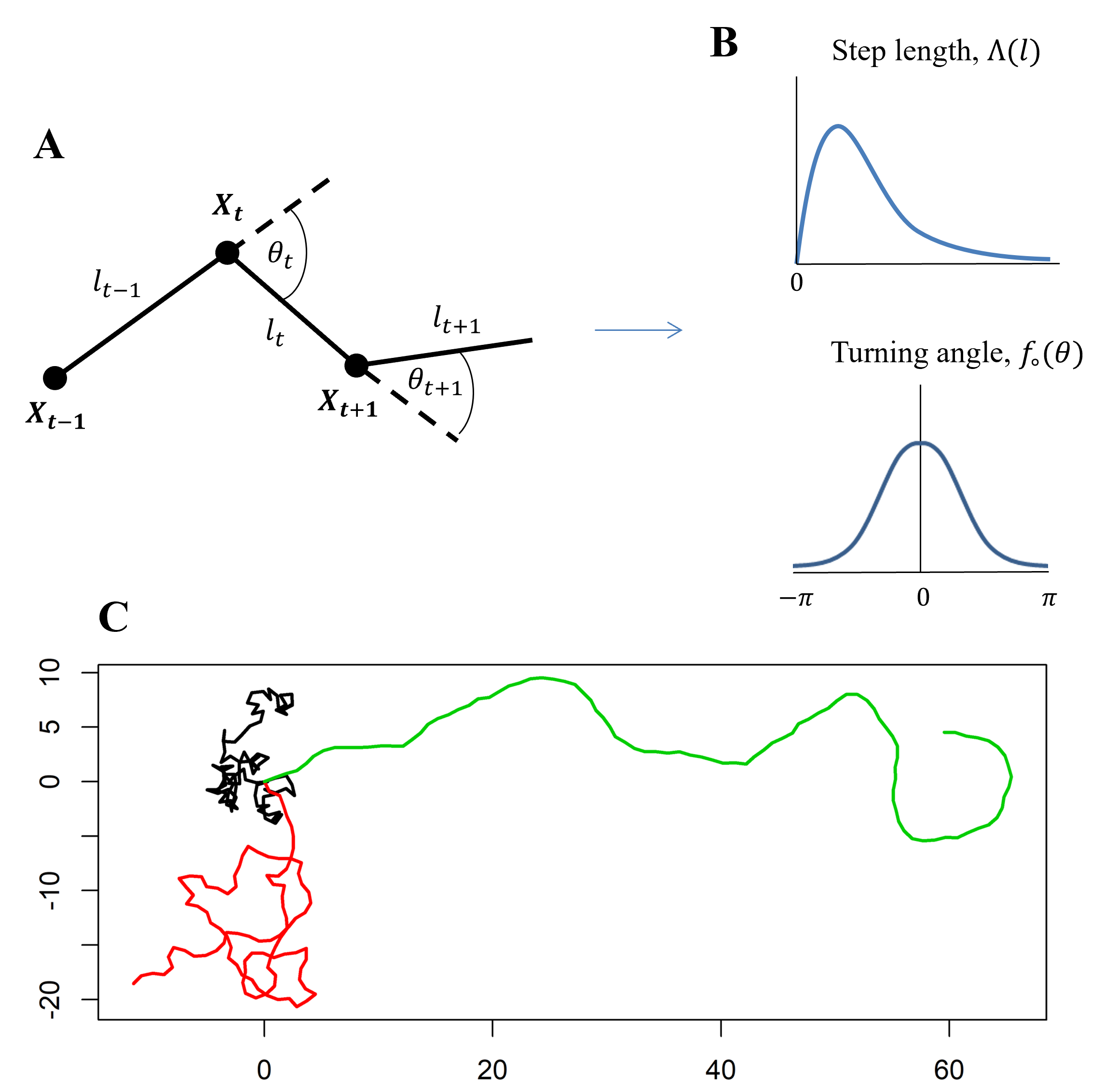}
    \caption{(A) shows the step-turn description of a discrete movement path featuring step-lengths, $\ell$, and turning angles, $\theta$, and locations $X$. (B) depicts the resulting step-length and turning angle distributions.  (C) shows three CRWs formed using a constant step-length distribution and turning angles drawn from a wrapped normal (WN) distribution, each trajectory was formed with a different value for the concentration parameter for the WN distribution, $\rho$, with: black - $\rho=0$, equivalent to random movement; red - $\rho=0.7$; green - $\rho=0.95$. Figure adapted from \citep{bailey2023assessment}.}
    \label{fig:CRW examp}
\end{figure}

Sub-sampling a RW is found by choosing a subset of the original locations in the initial RW, usually those which are a set number of time-steps apart (e.g. every $r$th point along the RW). Such a sub-sampled RW can also be described by using a step-turn description. Let $(X)_n$ be a CRW with turning angles, $\theta_{X_t}$, drawn from the circular distribution $f_\circ$ and step-lengths, $\ell_t$, drawn from a distribution on the positive reals $\Lambda$.  Consider the sequence $\left(X\right)^{\left\{ r\right\} }$ formed by sub-sampling the CRW given by $\left(X\right)_{n}$ at every $r$th point (Fig. 2), that is $\left(X\right)^{\left\{ r\right\} }=\left\{ X^{\{r\}}_{1},X^{\{r\}}_{2},\ldots,X^{\{r\}}_{t},\ldots\right\} =\left\{ X_{1},X_{1+r},\ldots,X_{1+tr},\ldots\right\} $. This has step-turn description given by $\left(\ell^{\{r\}}_i,\theta^{\{r\}}_{X_i}\right)$ where,
\begin{equation}
\ell^{\{r\}}_{i}=\Vert\overrightarrow{X^{\{r\}}_{i}X^{\{r\}}_{i+1}}\Vert=\Vert\overrightarrow{X_{i}X_{i+r}}\Vert=\sqrt{(x_{i+r}-x_{i})^{2}+\left(y_{i+r}-y_{i}\right)^{2}}
\end{equation}
and 
\begin{equation}
 \theta^{\{r\}}_{X_i}= \angle\left(\overrightarrow{X^{\{r\}}_{i-1}X^{\{r\}}_{i}},\overrightarrow{X_{i}^{\{r\}}X^{\{r\}}_{i+1}}\right) = \angle\left(\overrightarrow{X_{i-r}X_{i}},\overrightarrow{X_{i}X_{i+r}}\right) \Rightarrow \cos\theta_{X_i}=\frac{\overrightarrow{X_{i-r}X_{i}}\cdot\overrightarrow{X_{i}X_{i+r}}}{\Vert\overrightarrow{X_{i-r}X_{i}}\Vert\Vert\overrightarrow{X_{i}X_{i+r}}\Vert}
\end{equation}
with $\theta^{\{r\}}_{X_i}\sim f_\circ^{\{r\}}(\theta)$ and $\ell^{\{r\}}_{i}\sim\Lambda^{\{r\}}$.\\

\begin{figure}
    \centering
    \includegraphics[width=5.5in]{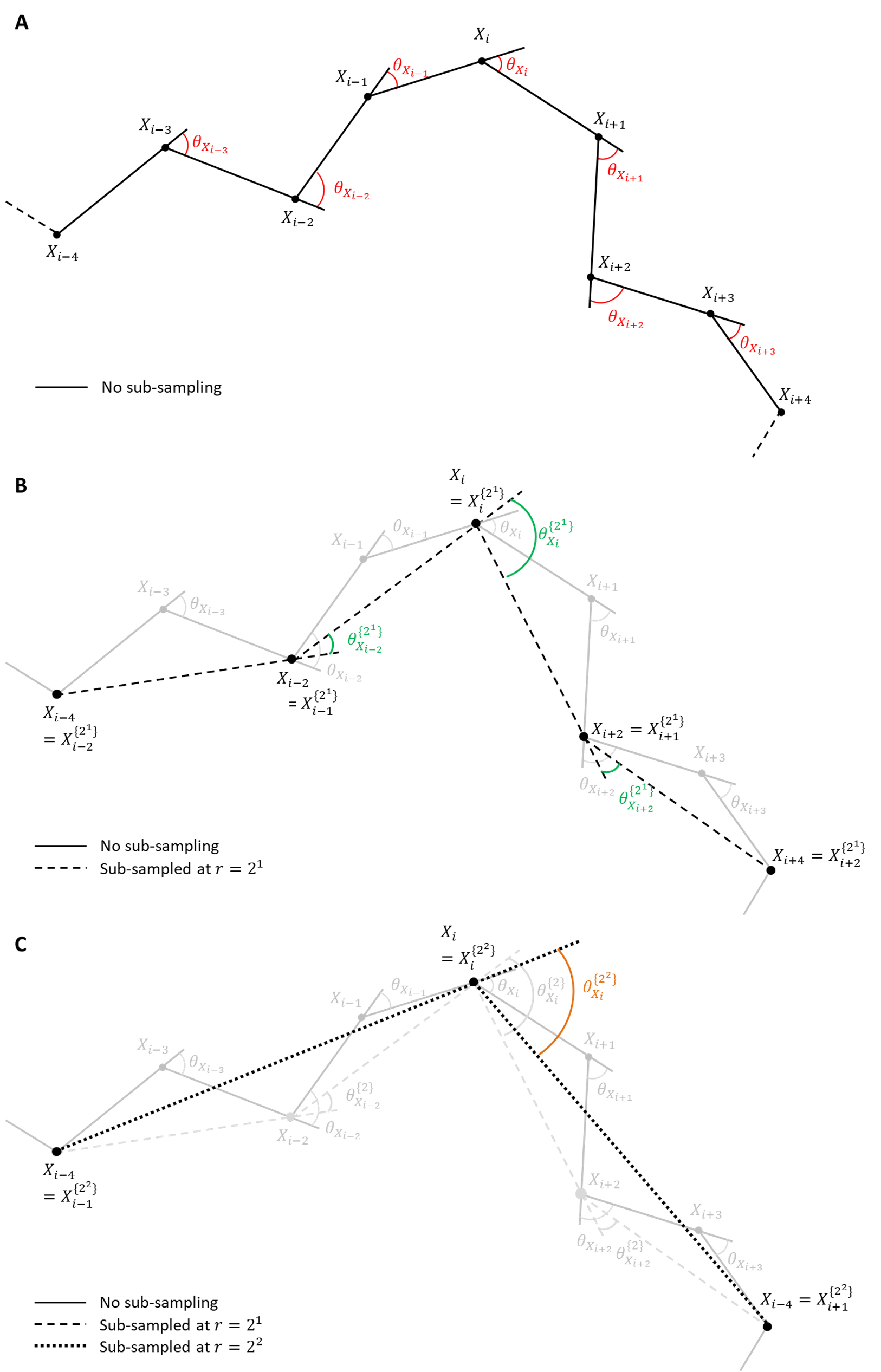}
    \caption{Examples of sub-sampling a RW, highlighting the resulting turning angles at each level of sampling. (A) is the initial, unsampled RW; (B) shows the initial RW sub-sampled at rate $r=2^1$; (C) shows the initial RW sub-sampled at rate $r=2^2$.}
\end{figure}

Whilst it is well known that as the sampling rate of a CRW increases, the resulting sampled RW appears closer to that of a SRW \citep{kareiva1983analyzing,codling2008random} and, hence, the turning angle distribution tends to the circular uniform distribution, it remains an open question to describe the connection between the distribution of the initial turning angles, $f_\circ(\theta)$, with the distribution of the turning angles sampled at every $r$th step, $f_\circ^{\{r\}}(\theta)$, and similar for the step lengths distributions $\Lambda(\ell)$ and $\Lambda^{\{r\}}(\ell)$.

\section{Main Claim}
Here we introduce the main claims of the work.  First, we consider a CRW with step-lengths of a fixed size, which will be sub-sampled at rate $r=2$, giving precise analytical solutions for the resulting turning angle and step-length distributions of the sub-sampled process.  We then move to the more general case, looking at a CRW sampled at rate $r=2^{k}$, where $k\in\mathbb{Z}_{\geq2}$.

\subsection{Case for sampling rate $r=2$, with initial CRW having step-lengths of fixed size.}

\begin{theorem}
Let $(X)_n=\{X_0, X_1, X_2,\ldots,X_n\}$ be a correlated random walk in $\mathbb{R}^2$ with step-turn characterisation $(\ell_i,\theta_{X_t})$, where $\ell_{t}=\ell\in\mathbb{R}_{>0}$, a fixed positive real value for all $t$, and $\theta_{X_t}\sim f_{\circ}$ where $f_{\circ}$ is a symmetric circular distribution centred around $0$.  The subsequent RW formed by sub-sampling $(X)_n$ at rate $r=2$, has circular turning-angle distribution $f^{\{2\}}_{\circ}$ given by
\begin{equation}\nonumber
    f^{\{2\}}_{\circ}(\theta)=\sum_{m=-1}^{1}f^{\{2\}}(\theta+2m\pi),
\end{equation}
where 
\begin{equation}\label{eq:theorem 1}
    f^{\{2\}}(\theta)=4f_{\circ}(\theta)\ast f^{*2}_{\circ}(2\theta),
\end{equation} and $f^{*2}$ is the power convolution of $f$ (i.e. $f^{*2}\stackrel{\text{def}}{=}f\ast f$), with $f_\circ(2\theta)=0$, for all $\theta\notin[-\frac{\pi}{2},\frac{\pi}{2}]$.
\end{theorem}

\begin{proof}
Let $(X)_{n}$ be a correlated random walk as in Theorem 1. For a sampling rate of $r=2$, every other point of $(X)_{n}$ will be chosen to form $(X)^{\{2\}}$
which will have step lengths $\ell^{\{2\}}_{t}$ drawn from some distribution
$\Lambda^{\{2\}}$ and turning angles $\theta^{\{2\}}_{X_t}$ drawn from some distribution $f^{\{2\}}_{\circ}(\theta)$.

\begin{figure}
    \centering
    \includegraphics[width=4.5in]{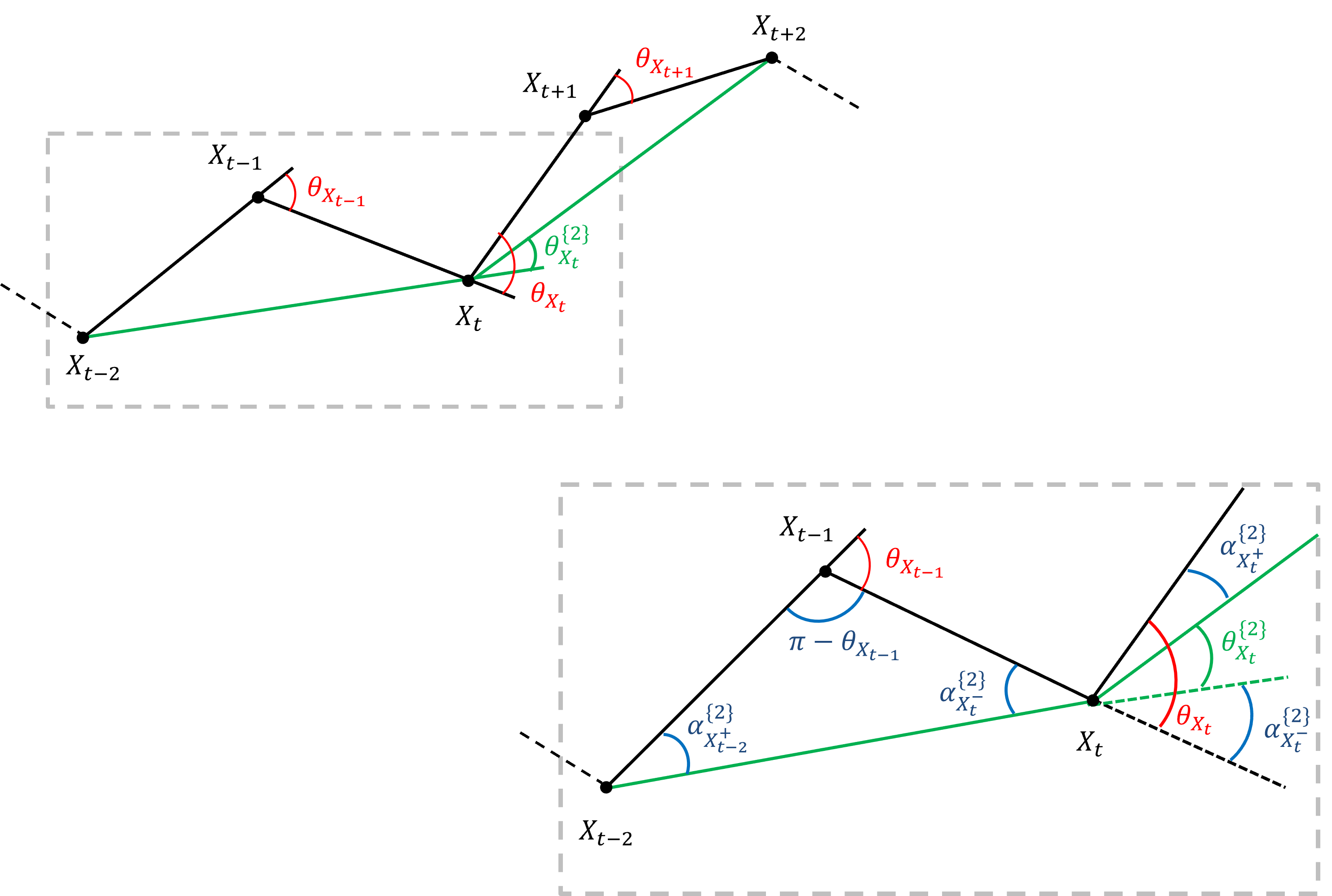}
    \caption{Figure depicting the angles required in the calculation of Theorem 1.  The green line depicts the sub-sampled RW formed at
    rate $r=2$.}
\end{figure}

Consider the point $X_t$, which w.l.o.g. we can assume is in both the initial and the sub-sampled walks.  The angle $\theta^{\{2\}}_{X_t}$ is therefore the angle between movement directions $\overrightarrow{X_{t-2}X_t}$ and $\overrightarrow{X_{t}X_{t+2}}$ (see Fig 2 \& 3).  Now let $T_t$ be the triangle formed by the point $X_t$ and the two previous locations (i.e. $T_t=\bigtriangleup X_tX_{t-1}X_{t-2}$).  This must be isosceles as the distances between successive points, are constant  ($||X_{t-2}, X_{t-1}||=||X_{t-1}, X_{t}||=\ell$).  Denote the internal angle of this triangle at $X_t$ as $\alpha^{\{2\}}_{X^-_t}$, that is $\alpha^{\{2\}}_{X^-_t}=\angle \left(X_{t-2}X_tX_{t-1}\right)$, where the $\{-\}$ in the subscript denotes that this is the internal angle formed from the \textit{prior} two locations at $X_t$.  Similarly, we denote the internal angle of the triangle $T_t$ at $X_{t-2}$ as $\alpha^{\{2\}}_{X^+_{t-2}}$ (Fig 2) where the $\{+\}$ denotes that this is the internal angle at $X_{t-2}$ and the \textit{succeeding} two points (Fig 3).  As $T_t$ is isosceles the internal angle at $X_{t-2}$ must equal the internal angle at $X_{t}$, hence $\alpha^{\{2\}}_{X^+_{t-2}}=\alpha^{\{2\}}_{X^-_{t}}$.  The remaining internal angle at $X_{t-1}$ is simply $\left(\pi-\theta_{X_{t-1}}\right)$, therefore we have
\begin{align}\nonumber
\alpha^{\{2\}}_{X^+_{t-2}}+\alpha^{\{2\}}_{X^-_t}+(\pi-\theta_{X_{t-1}})&=  \pi\\\nonumber
\Rightarrow \alpha^{\{2\}}_{X^-_t}+\alpha^{\{2\}}_{X^-_t}+(\pi-\theta_{X_{t-1}})&=  \pi,\\\nonumber
\text{hence,} \quad 2\alpha^{\{2\}}_{X^-_t}&=  \theta_{X_{t-1}}\\
\Rightarrow\alpha^{\{2\}}_{X^-_t}&=\frac{1}{2}\theta_{X_{t-1}}
\label{eqn: 6}
\end{align}

Repeating this approach for the triangle formed by $X_t$ and the two subsequent locations, $X_{t+1}, X_{t+2}$, gives 
\begin{equation}
    \alpha^{\{2\}}_{X^+_t}=\frac{1}{2}\theta_{X_{t+1}},
    \label{eqn: 6b}
\end{equation}
where $\alpha^{\{2\}}_{X^+_t}=\angle \left(X_{t+1}X_{t}X_{t+2}\right)$.
From Figure 3 and Eq \ref{eqn: 6}-\ref{eqn: 6b},  we can see that 
\begin{align}\nonumber
\theta^{\{2\}}_{X_t}&=\pm\theta_{X_t}\pm\alpha^{\{2\}}_{X^-_t}\pm\alpha^{\{2\}}_{X^+_t}\\
&=\pm\theta_{X_t}\pm\frac{1}{2}\theta_{X_{t-1}}\pm\frac{1}{2}\theta_{X_{t+1}}
\label{eqn: 7}
\end{align}

where each $\pm$ depends on the geometry of the RW and how the angles overlap in Fig 3.\\

As $\theta_{X_t},\theta_{X_{t-1}},\theta_{X_{t+1}}$ are i.i.d (see `Background'), drawn from the distribution $f_\circ$, Eq. \ref{eqn: 7} demonstrates that the distribution of $\theta^{\{2\}}_{X_t}$ can be found by convolution.\\

\begin{rem}
If $X$ is a random variable drawn from the continuous distribution $f_X$ with domain $[X_{\text{min}},X_{\text{max}}]$, then the distribution of $aX$, where $a\in\mathbb{R}$, is given by 
\begin{equation}
    aX\sim\frac{1}{|a|}f_X\left(\frac{1}{a}x\right),\quad \text{with domain }[aX_{\text{min}},aX_{\text{max}}].
    \label{eqn: aX dist}
\end{equation}
\label{rem:1}
\end{rem}

\noindent As $f_{\circ}$ is symmetrically distributed around 0, giving $f_{\circ}(\theta)=f_{\circ}(-\theta)$, Remark 1 indicates that the parity of each term in the RHS of Eq \ref{eqn: 7} will not affect the distribution of that term.  Therefore, denoting the distribution of $\theta^{\{2\}}_{X_i}$ as $f^{\{2\}}$, and using Remark \ref{rem:1} gives

\begin{align} \nonumber
f^{\{2\}}(\theta)=&f_{\circ}(\theta)\ast 2f_{\circ}(2\theta)\ast 2f_{\circ}(2\theta)\\
=&4f_{\circ}(\theta)\ast f^{*2}_{\circ}(2\theta) 
\end{align}
where $f^{*2}_{\circ}(2\theta)\stackrel{def}{=}f_{\circ}(2\theta)\ast f_{\circ}(2\theta)$.  Note that we define $f_\circ(2\theta)=0$ for all $\theta\notin[-\pi/2,\pi/2]$, and therefore $f^{\{2\}}(\theta)$ has domain $[-2\pi,2\pi]$.  Finally, we transform $f^{\{2\}}(\theta)$ into a circular distribution with domain $\theta\in[-\pi,\pi]$ by wrapping around the unit circle and denoting as $f^{\{2\}}_{\circ}$, to give
\begin{equation}
    f^{\{2\}}_{\circ}(\theta)=\sum_{m=-1}^{1}f^{\{2\}}(\theta+2m\pi).
\end{equation}

\end{proof}

\subsubsection{Simulation and Verification of Theorem 1}
To demonstrate Theorem 1, we can compare the PDFs predicted by the theorem with those found by repeated simulation of subsampled CRWs.  Fig \ref{fig: 4} shows these results for various circular distributions and parameter values, with predicted values from Theorem 1 as points and the simulated results as solid lines.  These demonstrate that the results predicted from Theorem 1 match the simulated results as expected.  

The simulated results were found by generating a CRW featuring $10^8$ steps, with each step-length being of unit length.  Turning angles were drawn from a zero-centred circular distribution with varying concentration parameters.  The subsequent CRWs were then sub-sampled at a rate of $r=2$ and the resulting turning angle distributions for the sub-sampled RW were found by calculating each new turning angle.  Results were then averaged over 100 repeated simulations for each distribution and parametrisation.  The predicted results from Theorem 1 were found via repeated use of Eq 8, where for each term in Eq 8, $10^8$ random values were drawn from the circular distribution, then summed modulo $2\pi$, with $f^{\{2\}}_{\circ}$ given by the distribution of these angles. In all cases, all calculations were carried out in `\textit{R}'.

\begin{figure}
    \centering
    \includegraphics[width=1\linewidth]{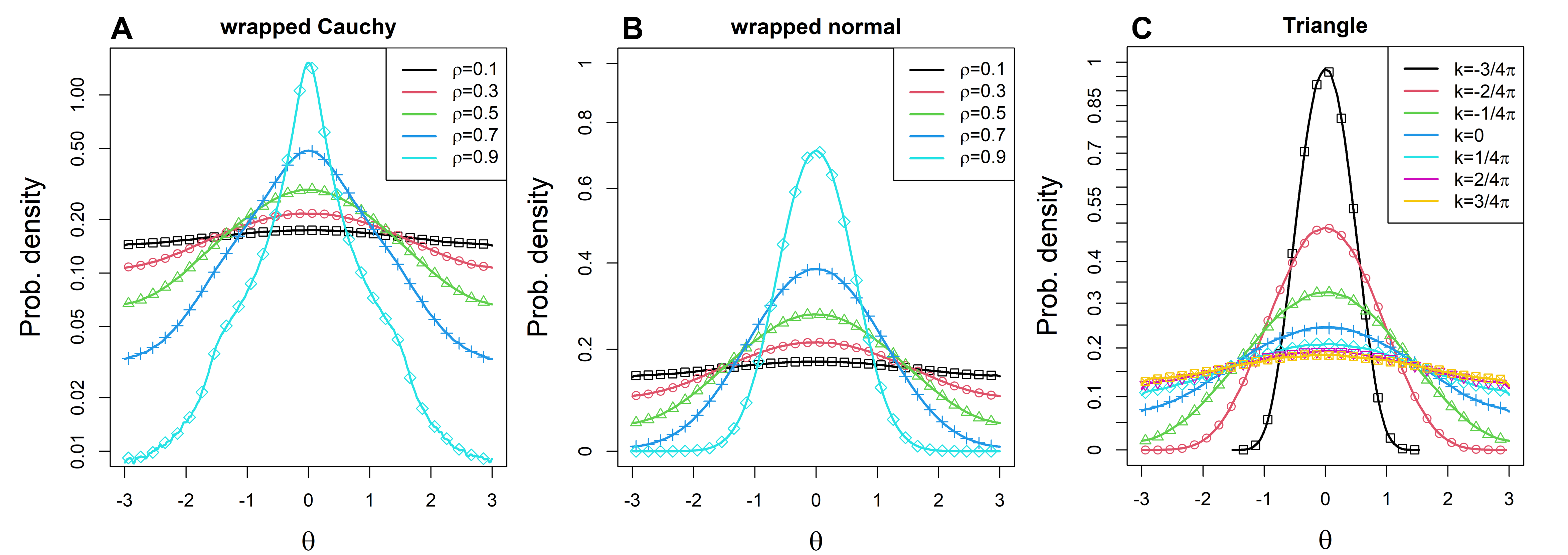}
    \caption{Results of turning angle distributions when a CRW is sampled at a rate $r=2$.  In all cases points are predicted values from Theorem 1 (specific values shown for $\theta$ are from $-\pi$ to $\pi$ at 0.2 intervals), solid lines are the simulated results.  (A) had initial distribution $f_{\circ}$ given by a wrapped Cauchy distribution; (B) had initial distribution $f_{\circ}$ given by a wrapped normal distribution; (C) had initial distribution $f_{\circ}$ given by a wrapped triangle distribution.  For each distribution the value of the parameter determining the peakedness of the distribution (See Appendix A) varied as in the provided key.  Note in all cases the $y$-axis has been log transformed to improve visualisation}
    \label{fig: 4}
\end{figure}

\subsubsection{Step-length distribution for subsampled RW from Theorem 1}
As a consequence of determining the turning angle distribution $f^{\{2\}}_{\circ}$, the step-length distribution $\Lambda^{\{2\}}$, can also be precisely calculated:

\begin{corollary}
    For a correlated random walk with constant step-length distribution, that is, for all $\ell$ we have $\Lambda(\ell)=\ell^*\in \mathbb{R}$, and turning angles, $\theta_t$, drawn from a circular distribution $f_\circ$ centred at 0, the distribution of step-lengths of the random walk formed by sub-sampling at rate $r=2$ is given by \begin{equation}
        \Lambda^{\{2\}}(\ell^{\{2\}})=\frac{2}{\sqrt{(\ell^*)^2-\left(\frac{\ell^{\{2\}}}{2}\right)^2}}f_\circ\left(2\arccos{\left(\frac{\ell^{\{2\}}}{2\ell^*}\right)}\right)
    \end{equation} which has range $[0,2\ell^*]$.  
\end{corollary}
\begin{proof}
Consider the consecutive points of a CRW $X_{t-1},X_t,X_{t+1}$.  We have $|X_{t-1}X_t|=|X_tX_{t+1}|=\ell^*$ by the properties of $\Lambda$, we are interested in describing the distribution of $\ell_t^{\{2\}}=|X_{t-1}X_{t+1}|$.  Using the cosine rule, The length of the vector between points $X_{t-1}$ and $X_{t+1}$ for all $t$ is simply
   \begin{align}\nonumber
     \lvert X_{t-1}X_{t+1}\rvert^2&=\lvert X_{t-1}X_t\rvert^2+\lvert X_{t}X_{t+1}\rvert^2-2\lvert X_{t-1}X_t\rvert\lvert X_{t}X_{t+1}\rvert\cos{\left(\pi-\theta_{X_t}\right)}\\\nonumber
       \left(\ell_t^{\{2\}}\right)^2&=2\ell^{*2}+2\ell^{*2}\cos{\left(\theta_{X_t}\right)}\\
       \Rightarrow \ell_t^{\{2\}}&=\sqrt{2\ell^{*2}(1+\cos{\left(\theta_{X_t}\right)})}=2\ell^*\cos{\frac{\theta_{X_t}}{2}}
       \label{eqn: 12}
   \end{align}
This shows that the distribution of $\ell^{\{2\}}$ is given in terms of the distribution of $\theta$, with $\ell_t^{\{2\}}$ taking values in $[0,2\ell^*]$ as $\theta_{X_t}\in(-\pi,\pi]$.  Now, as all $\theta_{X_t}$ are i.i.d, to find the distribution, $\Lambda^{\{2\}}(\ell^{\{2\}})$, we have
   \begin{align}\nonumber
       F_{\Lambda^{\{2\}}}(\ell^{\{2\}})&=P(L\leq \ell^{\{2\}})\\\nonumber
       &=P\left(2\ell^*\cos{\frac{\theta}{2}}\leq \ell^{\{2\}}\right)\\\nonumber
       &=2P\left(\theta\leq 2\arccos{\left(\frac{\ell^{\{2\}})}{2\ell^*}\right)}\right)\\
       &=2F_\theta\left(2\arccos{\left(\frac{\ell^{\{2\}})}{2\ell^*}\right)}\right)
   \end{align}
   where $F_{\Lambda^{\{2\}}}, F_\theta$ are the CDFs of $\Lambda^{\{2\}}$ and $f_\circ$ respectively.  The coefficient of 2 appears due to $\theta$ having domain $[-\pi,\pi]$, arccos having the restricted range of $[0,\pi]$ and cosine being an even function.  Hence, the distribution $\Lambda^{\{2\}}$ is given by
   \begin{align}\nonumber
       \Lambda^{\{2\}}(\ell^{\{2\}})=\frac{dF_{\Lambda^{\{2\}}}(\ell^{\{2\}})}{d\ell^{\{2\}}}&=2\frac{dF_\theta}{d\theta}\Bigg\lvert\frac{d\theta}{d\ell^{\{2\}}}\Bigg\rvert\\\nonumber
       &=2\frac{dF_\theta}{d\theta}\Bigg\lvert\frac{d}{d\ell^{\{2\}}}\left(2\arccos{\left(\frac{\ell^{\{2\}}}{2\ell^*}\right)}\right)\Bigg\rvert\\
       &=\frac{2}{\sqrt{\ell^{*2}-\left(\frac{\ell^{\{2\}}}{2}\right)^2}}f_\circ\left(2\arccos{\left(\frac{\ell^{\{2\}}}{2\ell^*}\right)}\right),
   \end{align}
Noting that here we have $\frac{dF_\theta}{d\theta}=f_\circ(\theta)$ and $\theta=2\arccos{\left(\frac{\ell^{\{2\}})}{2\ell^*}\right)}$.
\end{proof}
To demonstrate Corollary 1, Fig 5 plots the predicted results using the analytical expression for $\Lambda^{\{2\}}$ against the results found by simulations.  Results were found by simulating a CRW of length $10^8$ steps, where each step-length was of unit length and turning angles were drawn from a circular distribution.  The simulated step-length distributions (lines) were then found by calculating the distance between subsequent points when the RW is sampled at a rate $r=2$ and subsequently fitting a density distribution to these step-lengths.  The predicted results (points) were given by Eq. 12 from Corollary 1.  The figures demonstrate the validity of the Corollary.

\begin{figure}
    \includegraphics[width=\linewidth]{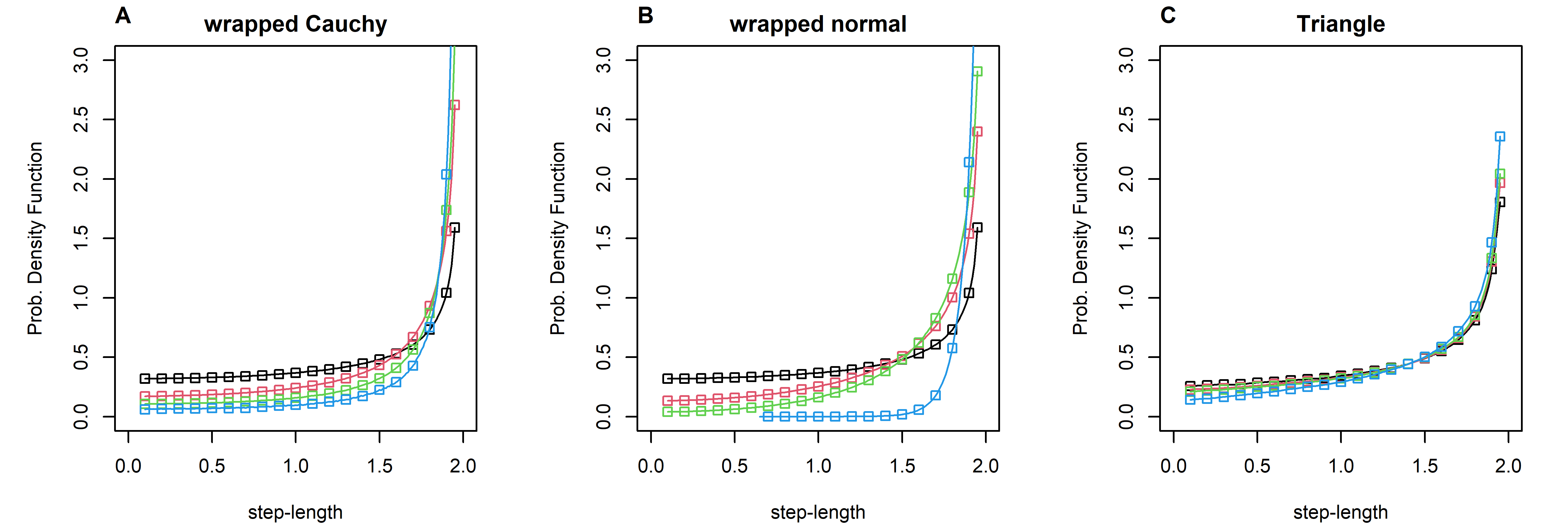}
    \caption{Demonstration of Corollary 1, showing the distribution of step-lengths after subsampling a random walk with fixed step lengths ($\ell^*=1$) at a rate of $r=2$.  Turning angles were taken from zero-centred wrapped Cauchy (A) and wrapped normal distribution (B), in both cases the black, red, green \& blue lines correspond to mean resultant vector (MRV) of $\rho=0.1,0.3,0.5,0.8$ respectively. Turning angles in (C) were taken from a wrapped triangle distribution with parameter values $a=6.28, 3.5, 2.75, 1.6$ (black, red, green \& blue lines respectively), which were chosen to closely mimic the MRV for (A) and (B) (see Appendix A).  In all cases the simulated results are given by the solid lines and the results given by Corollary 1 are shown as points.  All calculations were carried out in `R'.}
\end{figure}

\subsection{Case for sampling rate of $r=2^k$ where $k\geq2$.}

We now look to describe the distribution of the turning angles in the RW formed by subsampling at a rate of $r=2^k$, for integers $k\geq2$.  We will do this by describing the turning angle $\theta_{X_t}^{\{2^k\}}$ in a similar method to that used for $\theta_{X_t}^{\{2\}}$ in the proof of Theorem 1.  

Then look to build up the angles from the sampled RW $\theta_{X_t}^{\{r\}}$ as a linear sum of angles from the initial unsampled RW so as to be of the form, $\theta_{X_t}^{\{r\}}=\sum_\tau a_{t+\tau}\theta_{X_{t+\tau}}$, where $a_{t+\tau}$ are constants.  In this linear sum, each $\theta_{X_{t+\tau}}$ appears once and as each $\theta$ are i.i.d with distribution $f_{\circ}(\theta)$, we find the distribution of $\theta_{X_t}^{\{r\}}$ by convoluting the distributions of the $a_{t+\tau}\theta_{X_t}^{\{r\}}$ terms.\\

\noindent We begin by introducing the notation we will use for the remainder of the work:
\begin{itemize}
\item $(X)^{\{2^{k}\}}$ is the RW formed by sub-sampling the RW $(X)_n$ with sampling rate $r=2^{k}$.
\item $\theta^{\{2^k\}}_{X_t}$ is the turning angle at point $X_t$ for the subsampled RW $(X)^{\{2^{k}\}}$. i.e. the angle between $\overrightarrow{X_{t-2^k}X_t}$ and $\overrightarrow{X_tX_{t+2^k}}$.

\item $f_{\circ}^{\{2^{k}\}}$ is the circular distribution of the turning angles, $\theta_{X_t}^{\{2^k\}}$, for the sub-sampled random walk with rate $2^k$.
\item The intial turning angle distribution for the un-sampled random walk ($r=1\Leftrightarrow k=0$) is given by $f_{\circ}^{\{1\}}$, however, for sake of ease we drop the superscripts in this case.  i.e. $\theta_{X_t}^{\{1\}}=\theta_{X_t}$ and $f_{\circ}^{\{1\}}=f_{\circ}$.
\item $\alpha^{\{2^{k}\}}_{X^-_t}$, is the internal angle of the triangle formed at point $X_t$ and the \textit{previous} two locations of $X_{t-2^{k-1}}, X_{t-2^{k}}$; that is, \[\alpha^{\{2^{k}\}}_{X^-_t}=\angle \left(X_{t-2^k},X_t, X_{t-2^{k-1}}\right).\]
\item $\alpha^{\{2^{k}\}}_{X^+_t}$, is the internal angle of the triangle formed at point $X_t$ and the \textit{subsequent} two locations of $X_{t+2^{k-1}}, X_{t+2^{k}}$; that is, \[\alpha^{\{2^{k}\}}_{X^+_t}=\angle \left(X_{t+2^k},X_t, X_{t+2^{k-1}}\right)\]
\end{itemize}

\begin{theorem} Consider a discrete time CRW in $\mathbb{R}^2$, defined by a zero-centred, symmetric, unimodal turning angle distribution given by $f_\circ(\theta)$ and step-length distribution, $\Lambda(\ell)$, given by any continuous distribution on the positive reals with finite variance.  The distribution of turning angles in the sub-sampled RW of rate $r=2^k$ for $k\geq 2$ is closely given by
\begin{equation}\nonumber
    f^{\{2^k\}}_{\circ}(\theta)=\sum_{m=-\infty}^{\infty}f^{\{2^k\}}(\theta+2m\pi)
\end{equation} where 
    \begin{equation}\label{eq:theorem 2}
        f^{\{2^k\}}(\theta)=\frac{(p-\frac{1}{2})^{k(1-2^{k+1})}}{2^k!(2^k-1)!}\Bigg[f_\circ\left(\frac{\theta}{2^k(p-\frac{1}{2})^k}\right) \Conv_{\tau=1}^{2^k-1} f^{*2}_\circ\left(\frac{\theta}{\tau(p-\frac{1}{2})^k}\right)\Bigg],
    \end{equation}
with $f^{*2}$ the power convolution of $f$ (i.e. $f^{*2}\stackrel{\text{def}}{=}f\ast f$), and $p\in\left(\frac{1}{2},1\right]$ a parameter representing the tortuosity of the RW.
\end{theorem}

\begin{proof}
Let us consider a RW which has been sampled at the rate $r=2^k$ (with $k\geq 2$).  At the point $X_t$, we wish to describe the turning angle $\theta^{\{2^k\}}_{X_t}$, (the angle between $\overrightarrow{X_{t-2^k}X_t}$ and $\overrightarrow{X_tX_{t+2^k}}$). Following the arguments within the proof of Theorem 1 (see Eq. \ref{eqn: 7}), we have at point $X_t$;
\begin{equation}
   \theta_{X_t}^{\{2^{k}\}}=\pm\theta_{X_t}^{\{2^{k-1}\}}\pm\alpha_{X^-_t}^{\{2^{k}\}}\pm\alpha_{X^+_t}^{\{2^{k}\}}.
   \label{eqn: pre-intro of p}
\end{equation}

If we let $p_k$ be the probability that the geometry of the RW requires a positive value in the RHS terms of Eq \ref{eqn: pre-intro of p}, then we can define $\mathbbm{1}_{1-p_k}$ as an indicator function, which takes the value $1$ with probability $1-p_k$ and $0$ with probability $p_k$, so that,
\begin{equation} \nonumber
(-1)^{\mathbbm{1}_{1-p_k}} =\left\{ \begin{array}{ll}
    1 & \text{with probability } p_k
    \\ 
    -1 & \text{with probability } 1-p_k .
\end{array} \right.
\end{equation}
Therefore we can write Eq \ref{eqn: pre-intro of p} as

\begin{equation}
   \theta_{X_t}^{\{2^{k}\}}=(-1)^{\mathbbm{1}_{1-p_k}}\theta_{X_t}^{\{2^{k-1}\}}+(-1)^{\mathbbm{1}_{1-p_k}}\alpha_{X^-_t}^{\{2^{k}\}}+(-1)^{\mathbbm{1}_{1-p_k}}\alpha_{X^+_t}^{\{2^{k}\}} .
   \label{eqn: 18}
\end{equation}
This value $p_k$ depends on the geometry of the underlying RW, which for a general CRW is known to be affected by both the step-length distribution and the turning angle distribution \citep{benhamou2004reliably}.  The subscript $k$ indicates that $p_k$ is dependent upon the sampling rate.  It is assumed that $p_k$ is constant for each term in Eq. \ref{eqn: 18}.\\

\begin{rem}
In the proof of Theorem 1, a fixed step length distribution allowed for the precise computation of the turning angles in the subsampled walk by forming isosceles triangles (see Fig 2). When sampling at higher rates and for other initial step-length distributions, this is no longer possible as the triangles formed following a similar process are not necessarily isosceles. However, when considering the RW to be long (tending to infinity), it is reasonable to assume, on average, successive steps are of similar size. Therefore, we assume the triangulation technique deployed in the proof of Theorem 1 holds true for all sampling rates and initial step-length distributions.  This approximation can be seen to be appropriate in Appendix B, where the step length distribution is shown to have a negligible effect on turning angle distributions as the RW process is sub-sampled.
\end{rem}

Therefore, proceeding with the assumption in Remark 2, we have the triangles $\bigtriangleup X_{t-2^k}X_{t-2^{k-1}}X_t$ and $\bigtriangleup X_{t+2^k}X_{t+2^{k-1}}X_t$ are isosceles and we obtain 
\begin{align}\nonumber
  \alpha_{X^-_t}^{\{2^{k}\}}=&\frac{1}{2}{\theta^{\{2^{k-1}\}}_{X_{t-2^{k-1}}}}\\
  \alpha_{X^+_t}^{\{2^{k}\}}=&\frac{1}{2}{\theta^{\{2^{k-1}\}}_{X_{t+2^{k-1}}}}.
  \label{eqn: 19}
\end{align}

Substituting into Eq. \ref{eqn: 18} gives 
\begin{equation} \label{eq_theta{2^k}}
 \theta_{X_t}^{\{2^{k}\}}=(-1)^{\mathbbm{1}_{1-p_k}}\theta_{X_t}^{\{2^{k-1}\}}+(-1)^{\mathbbm{1}_{1-p_k}}\frac{1}{2}{\theta^{\{2^{k-1}\}}_{X_{t-2^{k-1}}}}+(-1)^{\mathbbm{1}_{1-p_k}}\frac{1}{2}{\theta^{\{2^{k-1}\}}_{X_{t+2^{k-1}}}}
 \end{equation}
 
Iteratively applying Eq.~\ref{eq_theta{2^k}}  will yield a linear expression for $\theta_{X_t}^{\{2^{k}\}}$ in terms of $\theta_{X_{t+\tau}}^{\{1\}}=\theta_{X_{t+\tau}}$.  Then, similar to the proof of Theorem 1, the distribution for $\theta^{\{2^{k}\}}$ can then be found by convolution.

Note that $k$ iterations of Eq.~\ref{eq_theta{2^k}} are required to reduce the exponent of terms on the RHS from $2^k$ to 1. As each iteration will add/subtract a term of the form $2^{k-j}$ to the subscripts of $X_t$ (see Eq.\ref{eqn: 19}), (where $j=\{1,2,\ldots,k\}$), for terms of the form $X_{t+\tau}$, the maximum value $\tau$ can be is $\tau=2^{k-1}+2^{k-2}+\cdots+2+1=2^k-1$, and by the symmetry of Eq.~\ref{eq_theta{2^k}}, $\tau$ will have a minimum of $t-(2^k-1)$.  

We therefore have 
\begin{equation}
    \theta_{X_t}^{\{2^{k}\}}=\sum_{\tau=-(2^k-1)}^{2^k-1} a^{\{2^k\}}_{t+\tau} \theta_{X_{t+\tau}}
    \label{eqn: 21}
\end{equation}
where the coefficients $a^{\{2^k\}}_{t+\tau}$ are to be determined.

\begin{sidewaysfigure}
    \centering
    \includegraphics[width=7.5in]{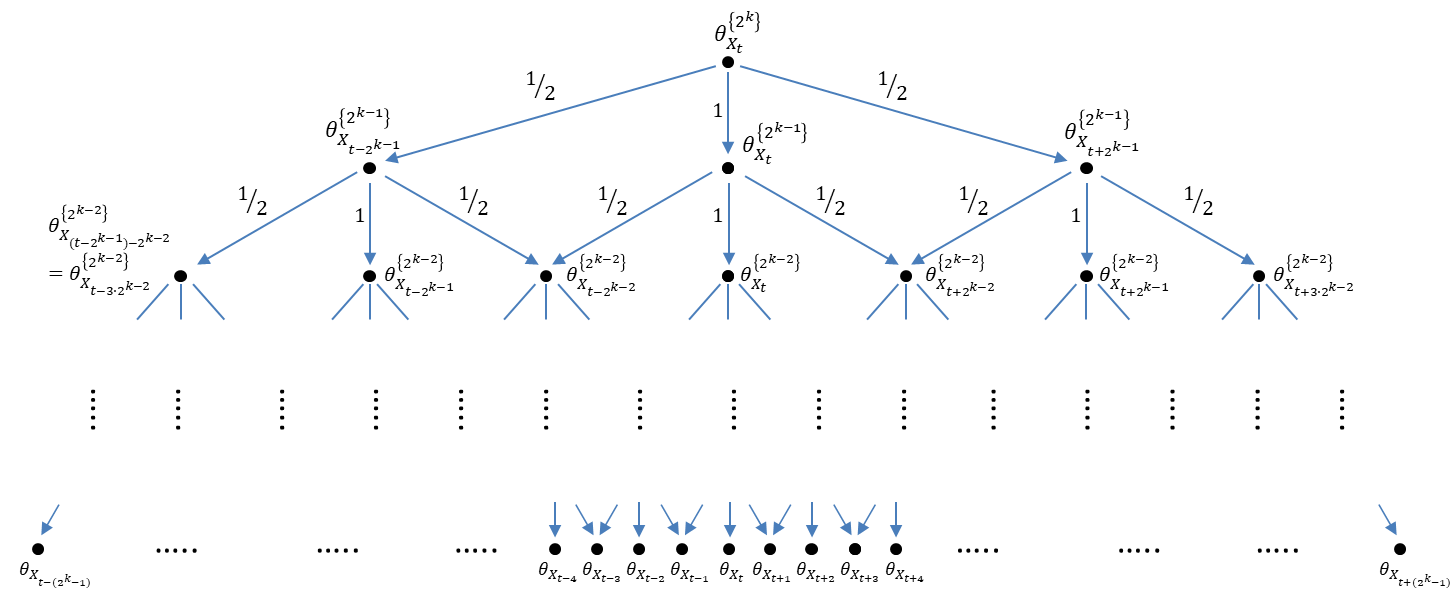}
    \caption{Directed graph showing the recursive relations given in Eq.~\ref{eq_theta{2^k}}. Summing the product of edge weights (sign dependent on probabilities $p_j$) along all paths from the root to a leaf gives the coefficient of that leaf's label in Eq.~\ref{eqn: 21}.}
    \label{fig:tree_turning_angles}
\end{sidewaysfigure}

To help calculate these coefficients, we depict the recursive relations given in Eq.~\ref{eq_theta{2^k}} as a directed graph (see Fig.~\ref{fig:tree_turning_angles}), where vertices correspond to the angle terms in the RHS of Eq. \ref{eq_theta{2^k}}, and outgoing edges from any vertex correspond to one iteration of Eq.~\ref{eq_theta{2^k}} on that vertex. Edge weights corresponding to the absolute value of the coefficients in Eq.~\ref{eq_theta{2^k}}. Note that each horizontal `level' only includes angles from the same sampling rate and as we must apply Eq.~\ref{eq_theta{2^k}} $k$ times, there will be precisely $k$ levels.  The final level occurs when no more iterations of Eq.~\ref{eq_theta{2^k}} can be applied, which is precisely when all $\theta^{\{r\}}$ have sampling rate $r=1$ and, are therefore from the initial turning angle distribution. For example, a sampling rate of $r=2$ applies one iteration of Eq.~~\ref{eq_theta{2^k}} and therefore would only require one level of Fig.~\ref{fig:tree_turning_angles} (see Appendix A for a further example with sampling rate $r=4=2^2$).\\

Recalling that for the coefficient of each term in Eq.~\ref{eq_theta{2^k}} there is a probability $p_j$ that it is positive and $1-p_j$ negative, where $p_j$ is dependent on $j$, the level of the graph (iteration of Eq.~\ref{eq_theta{2^k}}), with $j\in\{1,\ldots,k\}$.  Note, following convention we assume $j=0$ to be the top most level and $j=k$ to be the bottom level in Fig.~\ref{fig:tree_turning_angles}. The coefficient, $a^{\{2^k\}}_{t\pm\tau}$, for any given $\theta_{t\pm\tau}$, (which would be obtained via repeated iteration of Eq.~\ref{eq_theta{2^k}}), can then be calculated directly by considering all possible paths in Fig.~\ref{fig:tree_turning_angles} that start from the top vertex of $\theta^{\{2^k\}}_{X_t}$ and end at the vertex of $\theta_{t\pm\tau}$. For each possible path, the path weighting will be given by multiplying the weights on each individual edge along that path; this gives the coefficient of one particular occurrence of $\theta_{t\pm\tau}$ in the final expression (see Eq.~21), obtained via repeated iteration of Eq.~\ref{eq_theta{2^k}}. The overall value for the coefficient $a^{\{2^k\}}_{t\pm\tau}$ is found by summing over all the path weightings, for $\theta_{t\pm\tau}$.\\

As an example, we can see that the only path to $\theta_{t-(2^k-1)}$ is found by taking the left-most edge at each level, all of which have weights $(-1)^{\mathbbm{1}_{1-p_j}}\left(\frac{1}{2}\right)$.  As there are $k$ levels we have that $a^{\{2^k\}}_{t-(2^k-1)}=(-1)^{\mathbbm{1}_{1-p_1}+\mathbbm{1}_{1-p_2}+\cdots+\mathbbm{1}_{1-p_k}}\left(\frac{1}{2}\right)^k$ - and by symmetry the coefficient for $\theta_{t+(2^k-1)}$ has the same expression i.e. $a^{\{2^k\}}_{t+(2^k-1)}=(-1)^{\mathbbm{1}_{1-p_1}+\mathbbm{1}_{1-p_2}+\cdots+\mathbbm{1}_{1-p_k}}\left(\frac{1}{2}\right)^k$. \\

However, to calculate the coefficients without needing to consider every path into each vertex in the bottom level, we can consider how any vertex is connected to those in the level above it. 

For vertex $\theta_{X_{t\pm\tau}}^{\{2^{k-j}\}}$ in level $j$, we define $a^j_{t\pm\tau}$ to be the sum of all paths from the vertices in the higher level of $j-1$ to vertex $\theta_{X_{t\pm\tau}}^{\{2^{k-j}\}}$. Then

\begin{equation} \label{eq:a^j_exact}
    a^{j}_{t\pm\tau}=
\begin{cases}
    (-1)^{\mathbbm{1}_{1-p_j}}a^{j-1}_{t\pm\frac{\tau}{2}},\quad\tau\text{ even},\\
    \frac{1}{2}(-1)^{\mathbbm{1}_{1-p_j}}a^{j-1}_{t\pm(\frac{\tau -1}{2})},\quad \tau=2^{j}-1\\
    \frac{1}{2}\left[(-1)^{\mathbbm{1}_{1-p_j}}a^{j-1}_{t\pm\frac{\tau-1}{2}}+(-1)^{\mathbbm{1}_{1-p_j}}a^{{j-1}}_{t\pm\frac{\tau+1}{2}}\right],\; \text{else,}    
\end{cases}
\end{equation}
for $j=1,\dots, k$, with $a^0_{t\pm\tau}=1$.  Note, at the final level with $j=k$, the coefficients $a^{j}_{t\pm\tau}$ are precisely those of $a^{\{2^k\}}_{t\pm\tau}$ in Eq.~\ref{eqn: 21}.  
In Eq.~\ref{eq:a^j_exact}, the first case corresponds to those vertices in Fig~\ref{fig:tree_turning_angles} with one vertical arrow entering them (e.g. $\theta_{X_t}^{\{2^{k-2}\}}$); the second case corresponds to the extreme right/left branches with one left/right arrow entering them (e.g. $\theta_{X_{t-2^{k-1}}}^{\{2^{k-1}\}}$); and the third case is for those vertices which have two arrows entering in to them (e.g. $\theta_{X_{t-2^{k-2}}}^{\{2^{k-2}\}}$).

This demonstrates that for all $\tau\neq0$, the coefficient $a^{\{2^k\}}_{t\pm\tau}$ is a function of additions and subtractions of all reachable vertices above it, and is dependent on $p_1, p_2, \ldots,p_k$. However, since the values $p_1,\dots p_k$ are unknown, this expression becomes intractable.  Therefore, we assume $p_j=p$ for all $j=1,\dots, k$, and take the expectation of  $a^j_{t\pm\tau}$ in Eq.~\ref{eq:a^j_exact}.  These approximations gives new coefficients, $\tilde{a}^j_{t\pm\tau}$ given by

\begin{equation}  \label{eq:exp_a-tilde}
   \tilde{a}^{j}_{t\pm\tau}= \E\left[a^{j}_{t\pm\tau}\right]=
\begin{cases}
    (2p-1)\tilde{a}^{j-1}_{t\pm\frac{\tau}{2}},\quad\tau\text{ even},\\
    \frac{1}{2}(2p-1)\tilde{a}^{j-1}_{t\pm(\frac{\tau -1}{2})},\quad \tau=2^{j}-1\\
    \frac{1}{2}(2p-1)\left[\tilde{a}^{j-1}_{t\pm\frac{\tau-1}{2}}+\tilde{a}^{j-1}_{t\pm\frac{\tau+1}{2}}\right],\; \text{else}    
\end{cases}
\end{equation}

for $j=1,\dots, k$, and  $\tilde{a}^{0}_{t\pm\tau}=1$.  Where we have used 
\begin{equation}
    \E\left[(-1)^{\mathbbm{1}_{1-p}}\right]=(+1)p + (-1)(1-p)=(2p-1)
\end{equation} and assumed that $\left[(-1)^{\mathbbm{1}_{1-p}}\tilde{a}^{j}_{t\pm\frac{\tau-1}{2}}+(-1)^{\mathbbm{1}_{1-p}}\tilde{a}^{j}_{t\pm\frac{\tau+1}{2}}\right]$ is linear in respect to the expectation operator. \\

We now compute coefficients $\tilde{a}^{j}_{t\pm\tau}$, which replace ${a}^{j}_{t\pm\tau}$ in Eq.~\ref{eqn: 21}, to give a close approximation for $\theta_{X_t}^{\{r\}}$.  These coefficients can be found by repeatedly applying Eq.~\ref{eq:exp_a-tilde}, and can be written more compactly using the follwoing Claim.

\begin{claim}
The coefficients, $\tilde{a}^{\{2^{k}\}}_{t\pm\tau}$, are given by \begin{equation} \label{eq:claim_a-tilde}
    \tilde{a}^{\{2^k\}}_{t\pm\tau}=\tilde{a}^{k}_{t\pm\tau}=\frac{2^k-\tau}{2^k}(2p-1)^k
\end{equation}
where $k\geq 0$, and $\tau =0,1,\ldots,2^k-1$.
\label{claim:1}
\end{claim}

Claim \ref{claim:1} can be shown through induction. The case $k=0$ follows trivially from Eq.~\ref{eq:exp_a-tilde}, since $\tilde{a}^{\{1\}}_t=\tilde{a}^0_t=1$ and substituting $k=0$ in Eq.~\ref{eq:claim_a-tilde} gives $\tilde{a}^0_{t+\tau}=1-\tau$, where $\tau$ can only take value $(2^0-1)=0$.

We proceed with induction on $k$. Assume Eq. \ref{eq:claim_a-tilde} holds for $k$, now consider the expression for $k+1$. Working through each of the three cases separately, we firstly assume $\tau$ is even, case (i), then by Eq. \ref{eq:exp_a-tilde}
\begin{align}\nonumber
\tilde{a}^{\{2^{k+1}\}}_{t\pm\tau}=\tilde{a}^{k+1}_{t\pm\tau}
&=(2p-1)\tilde{a}^{k}_{t\pm\frac{\tau}{2}}
\\ \nonumber
&=(2p-1)\frac{2^k-\tau/2}{2^k}(2p-1)^k
\\ \nonumber
&= \frac{2^{k+1}-\tau}{2^{k+1}}(2p-1)^{k+1},
\end{align}
as required.  Next, case (ii), assume $\tau=2^{k+1}-1$ and note that for this, 
\begin{align} \nonumber
\frac{\tau-1}{2} &=\frac{2^{k+1}-2}{2}
\\ \nonumber
&=2^k-1.
\end{align}

Therefore for $\tau=2^{k+1}-1$ and by Eq. \ref{eq:exp_a-tilde}, we have
\begin{align} \nonumber
\tilde{a}^{\{2^{k+1}\}}_{t\pm\tau}=\tilde{a}^{k+1}_{t\pm\tau}
&= \frac{1}{2}(2p-1)\tilde{a}^{k}_{t\pm(\frac{\tau -1}{2})}
\\ \nonumber
&=\frac{1}{2}(2p-1) \frac{2^k-(\tau-1)/2}{2^k}(2p-1)^k
\\ \nonumber
&= \frac{2^k-(2^k-1)}{2^{k+1}}(2p-1)^{k+1}
\\ \nonumber
&= \frac{2^{k+1}-(2^{k+1}-1)}{2^{k+1}}(2p-1)^{k+1}
\\ \nonumber
&= \frac{2^{k+1}-\tau}{2^{k+1}}(2p-1)^{k+1},
\end{align}
as required.  Finally, we consider the case of odd $\tau<2^{k+1}-1$. Applying Eq. \ref{eq:exp_a-tilde},

\begin{align*}
\tilde{a}^{\{2^{k+1}\}}_{t\pm\tau}
&= \frac{1}{2}(2p-1)\left[\tilde{a}^{k}_{t\pm\frac{\tau-1}{2}}+\tilde{a}^{k}_{t\pm\frac{\tau+1}{2}}\right]
\\
&=\frac{1}{2}(2p-1)\left[ \frac{2^k- (\tau-1)/2}{2^k}(2p-1)^k
+ \frac{2^k-(\tau+1)/2}{2^k}(2p-1)^k
\right]
\\
&=\frac{1}{2^{k+1}}\left[ 2^k-\frac{\tau-1}{2}
+2^k-\frac{\tau+1}{2}
\right](2p-1)^{k+1}
\\
&=\frac{2^{k+1}-\tau}{2^{k+1}}(2p-1)^{k+1}. 
\end{align*}

Thus in all three cases, given the inductive hypothesis, Eq.~\ref{eq:claim_a-tilde} holds for $k+1$, therefore the claim holds for all $k\geq 0$.

We now note that, Eq.~\ref{eq:claim_a-tilde} can be written equivalently as: 
\begin{equation} \label{eq:claim_a-tilde_nopm}
        \tilde{a}^{\{2^k\}}_{t+\tau}=\frac{2^k-\lvert\tau\rvert}{2^k}(2p-1)^k,
\end{equation}
which allows $\tau$ to take integer values $\{-(2^k-1),\ldots,2^k-1\}$. 
Hence, using Eq.~\ref{eqn: 21} and Eq.~\ref{eq:claim_a-tilde_nopm} we can write the angle at $X_t$ of the subsampled random walk, $\theta_{X_t}^{\{2^k\}}$, as
\begin{align}
    \theta_{X_t}^{\{2^k\}}=&\sum_{\tau=-(2^k-1)}^{2^k-1} \tilde{a}^{\{2^k\}}_{t+\tau} \theta_{X_{t+\tau}}\\\nonumber
    =&\sum_{\tau=-(2^k-1)}^{2^k-1} \frac{2^k-\lvert\tau\rvert}{2^k}(2p-1)^k\theta_{X_{t+\tau}}\\
   =&\sum_{\tau=-(2^k-1)}^{2^k-1}(p-\frac{1}{2})^k(2^k-\lvert\tau\rvert)\theta_{X_{t+\tau}}. \label{eq:28}   
\end{align}

Therefore, as  $\theta_{X_t}^{\{2^k\}}$ can be written as a linear sum of the terms $\tilde{a}^{\{2^k\}}_{t+\tau} \theta_{X_{t+\tau}}$, where all $\theta_{X_{t+\tau}}$ are i.i.d,  the distribution of $\theta_{X_t}^{\{2^k\}}$, can be found by a convolution of the distributions for the $\tilde{a}^{\{2^k\}}_{t+\tau} \theta_{X_{t+\tau}}$ terms.

Let us denote the distribution of $\tilde{a}^{\{2^k\}}_{t+\tau} \theta_{X_{t+\tau}}$ as $f_\tau$, then using Eq \ref{eqn: aX dist} from Remark 1, and noting that all $\theta_{X_{t+\tau}}$ are modelled by the distribution $f_\circ(\theta)$, we have
\begin{equation}\label{eqn:29}
    f_\tau(\theta)=\frac{1}{\lvert(p-\frac{1}{2})^k\rvert\lvert2^k-\lvert\tau\rvert\rvert}f_\circ\left(\frac{\theta}{(p-\frac{1}{2})^k(2^k-\lvert\tau\rvert)}\right).
\end{equation}

 We simplify this expression by noting $|2^k-|\tau||= 2^k-|\tau|$, since $\tau$ only takes integer values in $\{-(2^k-1),\ldots,(2^k-1)\}$.  Similarly, as $|(p-\frac{1}{2})|^k$ is symmetric around $p=1/2$ for $p\in[0,1]$ and all non-negative $k$, we can restrict $p$ to be in $(\frac{1}{2},1]$ and so $|(p-\frac{1}{2})|^k= (p-\frac{1}{2})^k$. Note, this restriction on $p$ also means $(p-\frac{1}{2})^k>0$ inside the $f_\circ$ function in Eq.~\ref{eqn:29}, however since $f_\circ$ is a symmetric distribution around $0$ (i.e. $f_\circ(\theta)=f_\circ(-\theta)$), this restriction does not affect the convolution. 
 Therefore, we have
\begin{equation}
    f_\tau(\theta)=\frac{1}{(p-\frac{1}{2})^k(2^k-\lvert\tau\rvert)}f_\circ\left(\frac{\theta}{(p-\frac{1}{2})^k(2^k-\lvert\tau\rvert)}\right).
\end{equation}
Returning to Eq.~\ref{eq:28} and using the expression for $f_\tau$ given above, we have 
\begin{align}\nonumber
f^{\{2^{k}\}}(\theta)&=\Conv_{\tau=-(2^{k}-1)}^{2^{k}-1} f_{\tau}(\theta)\\\nonumber
&=\Conv_{\tau=-(2^{k}-1)}^{2^{k}-1}\frac{1}{(p-\frac{1}{2})^k(2^k-\lvert\tau\rvert)}f_\circ\left(\frac{\theta}{(p-\frac{1}{2})^k(2^k-\lvert\tau\rvert)}\right)\\\nonumber
&=\prod_{\tau=-(2^{k}-1)}^{2^{k}-1}\frac{1}{(p-\frac{1}{2})^k(2^k-\lvert\tau\rvert)}\Conv_{\tau=-(2^{k}-1)}^{2^{k}-1}f_\circ\left(\frac{\theta}{(p-\frac{1}{2})^k(2^k-\lvert\tau\rvert)}\right)\\
&=\left[\frac{1}{(p-\frac{1}{2})^k}\right]^{2\cdot(2^k-1)+1}\frac{1}{2^k\left[(2^k-1)!\right]^2}\Conv_{\tau=-(2^{k}-1)}^{2^{k}-1}f_\circ\left(\frac{\theta}{(p-\frac{1}{2})^k(2^k-\lvert\tau\rvert)}\right)\label{eq:31}
\end{align}
If we consider the expression inside the $f_\circ(\cdot)$ in Eq. \ref{eq:31}, we note that as $\tau$ takes integer values $\{-(2^k-1),\ldots,(2^k-1)\}$ and due to the absolute function, we have that $\tau$ takes every value from $1$ to $2^k-1$ twice with $\tau=0$ the only non-repeated term.  Therefore, the convolution in Eq.~\ref{eq:31} can be written using only the positive terms for $\tau$ and the power convolution, $f^{*2}$, along with the lone $\tau=0$ term.  This, and noting that $2^k\left[(2^k-1)!\right]^2=2^k!(2^k-1)!$, allows Eq. \ref{eq:31} to be written as 
\begin{equation}
f^{\{2^{k}\}}(\theta)=\frac{(p-\frac{1}{2})^{k(1-2^{k+1})}}{2^k!(2^k-1)!}\Bigg[f_\circ\left(\frac{\theta}{2^k(p-\frac{1}{2})^k}\right) \Conv_{\tau=1}^{2^k-1} f^{*2}_\circ\left(\frac{\theta}{\tau(p-\frac{1}{2})^k}\right)\Bigg]
\end{equation}

where each $f_\circ\left(\frac{1}{\tau(p-\frac{1}{2})^k}\theta\right)$ distribution is defined to be $0$ for all $\theta\notin [-\pi\tau(p-\frac{1}{2})^k,\pi\tau(p-\frac{1}{2})^k]$, similar to the case for $k=2$ in Theorem 1.

Finally, we denote $f^{\{2^k\}}_{\circ}$ as the circular distribution found by wrapping $f^{\{2^{k}\}}(\theta)$ around the unit circle, to give a distribution with domain $(-\pi,\pi]$; that is 
\begin{equation}
    f^{\{2^k\}}_{\circ}(\theta)=\sum_{m=-\infty}^{\infty}f^{\{2^{k}\}}(\theta+2m\pi),
\end{equation} 
as required.
\end{proof}

\subsubsection{Accuracy of Theorem 2}
We now demonstrate the accuracy of the expression for $f_{\circ}^{\{2^{k}\}}$ given in Theorem 2, by comparing it to simulated results, similar to Section 3.1.1 for Theorem 1.

Fig \ref{Fig: Theorem 2} compares the distributions of turning angles, when 
a CRW featuring fixed step-lengths of unit length and turning angle distribution given by either a wrapped normal or wrapped Cauchy, is sub-sampled at rates of $r=2$ (green), $r=4$ (cyan), $r=8$ (blue).  In each panel solid lines represent the distribution found by simulating the RW and the points show the predicted results found using Theorem 2 and Eq. \ref{eq:28}.  In each plot and for each sub-sampling rate, the value of the parameter $p$ was found by best-fitting the predicted curve (points) to the observed results (solid lines) - further discussion of this is given in section 3.3.

The results show an accurate fit, with only small discrepancies observable at the tails (values close to $\pm\pi$) of the sub-sampled distributions.  It should be noted that these results are sensitive to the precise value of the parameter $p$ which is explored further in the following section.

\begin{figure}   
    \centerline{\includegraphics[width=\linewidth]{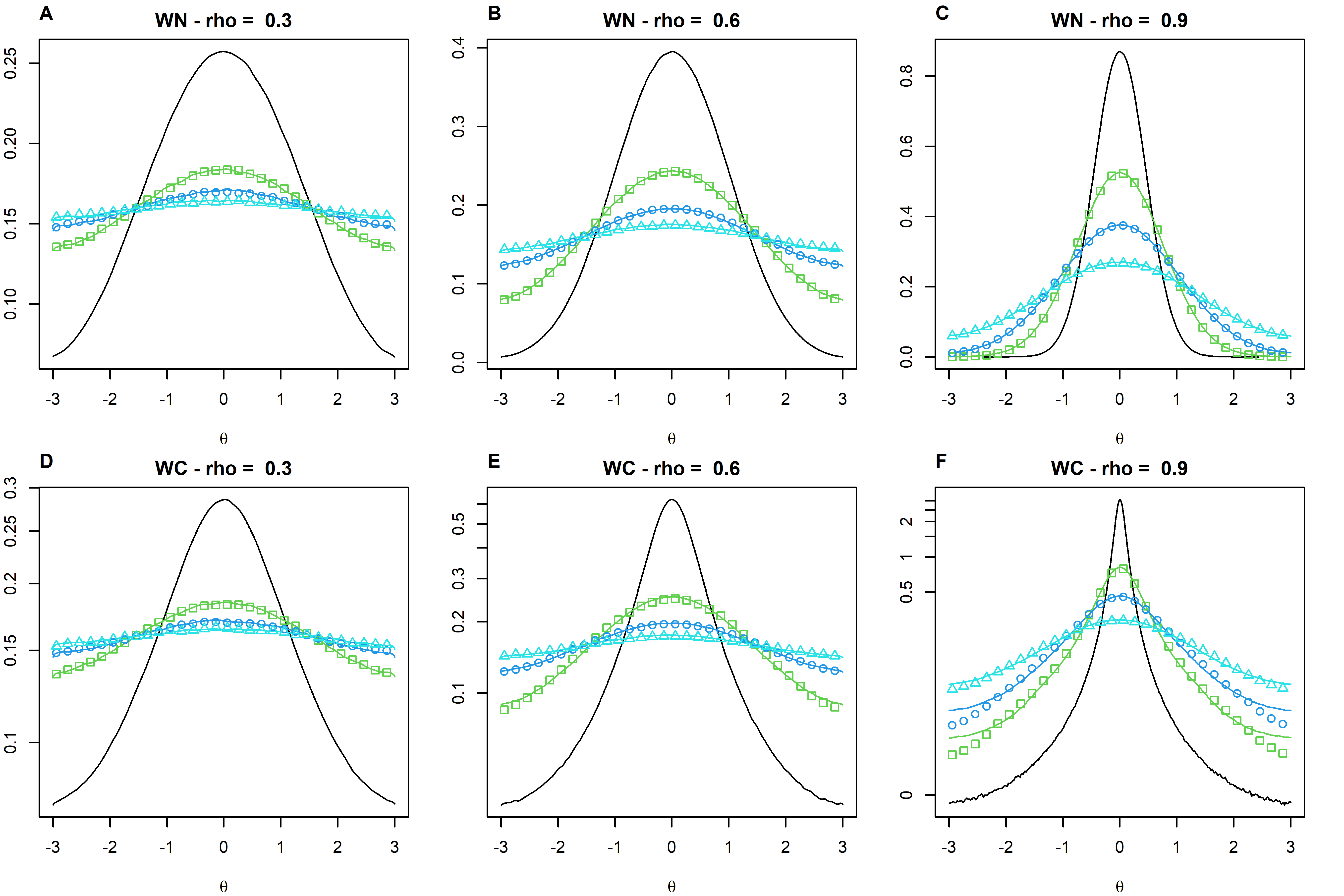}}
    \caption{Panel showing the predicted results from Eq.\ref{eq:theorem 2} in Theorem 2 (points) against simulated results (solid lines).  In all plots, the subsampling rates shown are for $r=2^k$ with $k=2$ (green), $k=3$ (blue), $k=4$ (cyan).  The initial distributions were a wrapped normal (A-C) and a wrapped Cauchy (D-F), with mean resultant vector length of $\rho=0.3$ (A \& D), $\rho=0.6$ (B \& E), and $\rho=0.9$ (C \& F).  In all cases the value for $p$ was best fitted for each subsampling and $\rho$ value.  The black solid line shows the initial distribution of turning angles in the RW.  Note, due to the peakedness of the wrapped Cauchy distribution, the $y$-axis has been log transformed to aid visualisation in plots (D-F).}
\label{Fig: Theorem 2}\end{figure}

\subsection{The parameter $p$.}

Here we briefly explore the tuning parameter $p$ as required in Theorem 2.  As discussed, the parameter $p$ is linked to the geometry of the underlying RW, which for a general CRW can be described by the `sinuosity' of the movement \citep{benhamou2004reliably}.  Therefore, the parameter $p$ must encode information about the initial step-length and turning angle distributions as well as accounting for auto-correlation inherited in the process from the sub-sampling \citep{nams2013sampling}.

This can be seen by considering sub-sampling at rate $r=2$ in Theorem 2 and comparing this to Theorem 1.  If we set $r=2$ and therefore $k=1$ in Eq.~\ref{eq:theorem 2} from Theorem 2, we should return Eq.~\ref{eq:theorem 1} from Theorem 1.  Doing so gives 
\begin{equation}\nonumber
    f^{\{2\}}(\theta)=\frac{1}{2(p-\frac{1}{2})^3}\left[f_\circ\left(\frac{\theta}{2(p-\frac{1}{2})}\right)\ast f_{\circ}^{\ast 2}\left(\frac{\theta}{(p-\frac{1}{2})}\right)\right].
\end{equation}
Setting $p=1$, gives
\begin{equation} \nonumber
    f^{\{2\}}(\theta)=4\left[f_\circ\left(\theta\right)\ast f_{\circ}^{\ast 2}\left(2\theta\right)\right].
\end{equation}
as required from Eq.~\ref{eq:theorem 1}.

This demonstrates that when the sampling rate is $r=2$ and the initial underlying step-length distribution is constant, as was the case for Theorem 1, we have  $p=1$.  For differing initial step-length distributions the precise value of $p$ may vary slightly due to following steps in the RW process not being of the same length (see Remark 2).  Whilst Appendix C demonstrates that the precise step-length distribution does not greatly influence the given turning angle distribution after sampling, the parameter $p$ would encode any slight effect that does occur.

\begin{figure}
    \centering
    \includegraphics[width=2.5in]{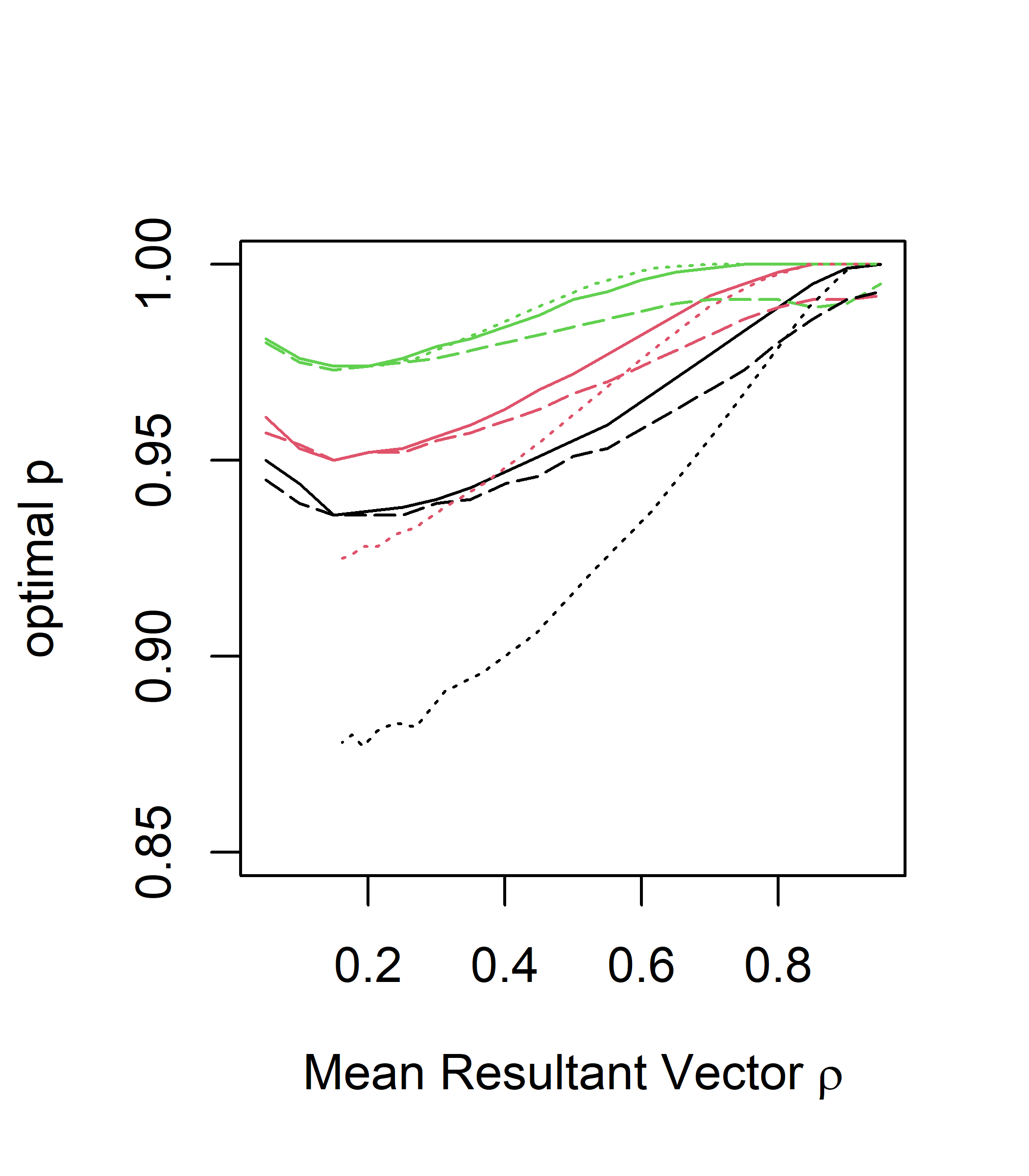}
    \caption{Optimal $p$ values as a function of the length of the mean resultant vector $\rho$.  Colours denote different subsampling levels: $k=2$ - green; $k=3$ - red; $k=4$ - black.  Line type denotes different distributions: solid - wrapped normal; dashed - wrapped Cauchy; dotted - wrapped triangle.}
    \label{fig:opt p}
\end{figure}

As the sub-sampling rate increases to find the optimal value of $p$ for a given, distribution, length of the mean resultant vector length $\rho$ and sampling rate, $r$, we manually best fitted distributions by minimising the MSD ($L^2$ distance) between the distributions for simulated results and those found by a parameter sweep over $p$,  similar to the approach taken in \citet{bailey2021emergence}.  This required comparing the observed turning angle distributions from sub-sampled RWs with those found by letting $p$ vary in Eq. \ref{eq:28}.

\begin{figure}
    \centering
    \includegraphics[width=\linewidth]{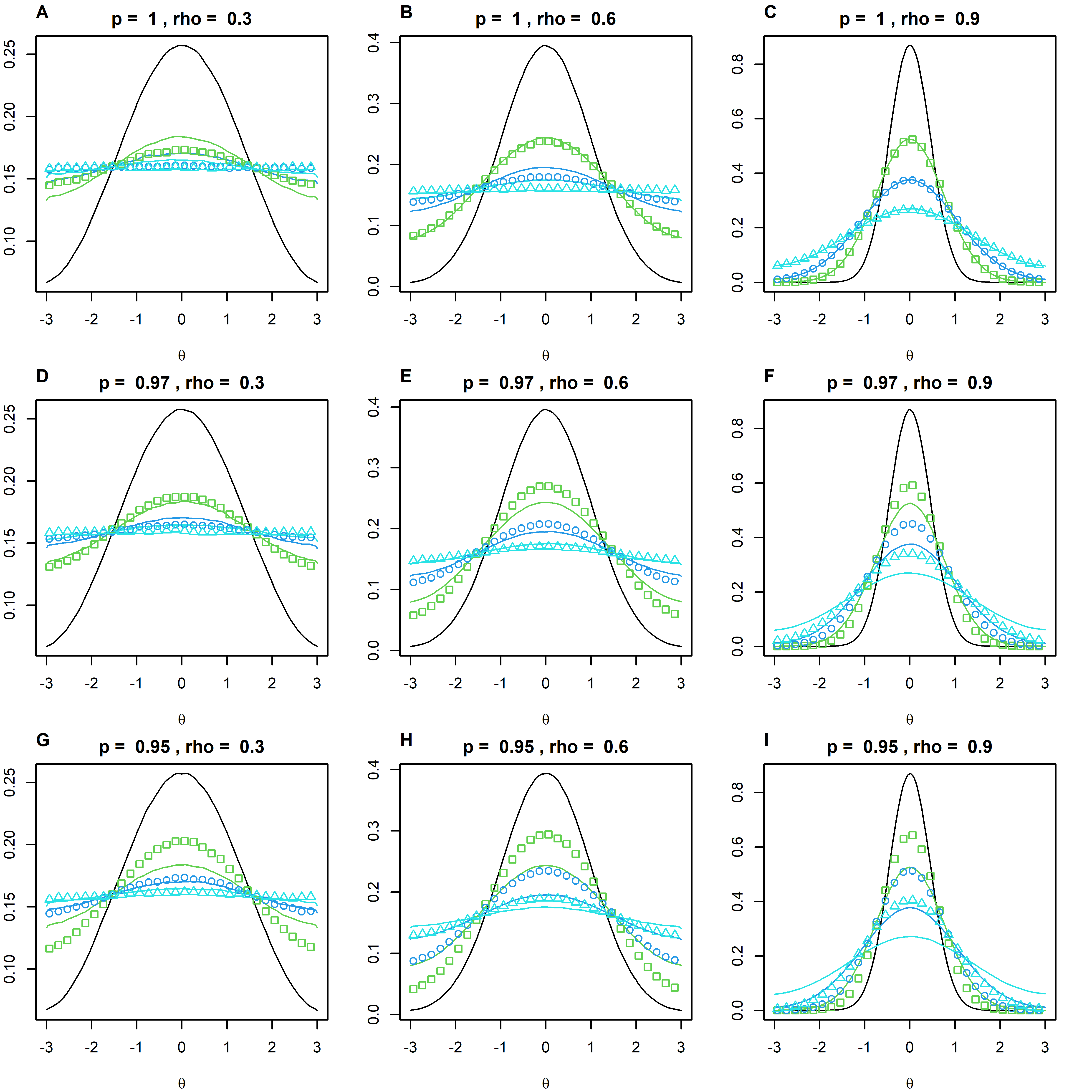}
    \caption{Panel demonstrating the sensitivity of the predicted distributions as given in Theorem 2 to changes in the value of the parameter $p$.  The value of $p$ was  fixed at $p=1$ (top row); $p=0.97$ (middle row); and $p=0.95$ (bottom row).  Results are for the initial CRW featuring turning angles described by a wrapped Normal distribution (black line) and fixed step-length distribution.  Colours correspond to differing levels of sub-sampling: green - $r=2$; cyan - $r=4$; dark blue - $r=8$.  In all cases, observed results are in solid lines, results given by Theorem 2 are points.}
    \label{fig:theorem2_fixedp}
\end{figure}

Results are shown in Fig.~\ref{fig:opt p}, and demonstrate that the value of $p$ is sensitive to the sampling rate (coloured lines) and to the underlying turning angle distribution (type of line).  Although it is noticeable that the wrapped normal and wrapped Cauchy give similar result, compared to the wrapped triangle.  Given that both the wrapped normal and wrapped Cauchy are two of the most utilised and observed distributions in movement data (both coming from the family of symmetric wrapped stable distributions \citep{jammalamadaka2001topics}; see Appendix A), this indicates that it is the sampling rate and concentration parameter that are the most important factors to consider when determining the optimal $p$ value.  Whilst it is also apparent that the value of $p$ mostly lies in $(0.94,1)$, the sensitivity of Theorem 2 to the precise value of $p$ is demonstrated in Fig.~\ref{fig:theorem2_fixedp}, which shows the sensitivity of the distributions to the parameter $p$ by fixing the value of $p$ and comparing the predicted (points) with the observed results (solid lines).  This shows that a relatively small perturbation away from the optimal value for $p$ of order $\pm0.02$ can have large effects on the accuracy of Theorem 2.

\section{Applications}
Due to their ease of use, their interpretability and natural fit to telemetry data, discrete-time continuous space random walks have been shown to be an important tool in the analysis of animal movement and in the field of movement ecology \citep{kareiva1983analyzing,nathan2008movement,codling2008random,bailey2018navigational,ModelingandDataAnalysisinMovementEcology}, driving the development of new theory \citep{codling2008random,smouse2010stochastic,lutscher2021correlated} and deepening our understanding of complex ecological processes \citep{berg1993random,hooten2017animal,lewis2013dispersal}.\\
Here we present two examples of how comparing the expected sub-sampled turning angle distributions to that found from observed data, can be used to qualitatively analyse movement data. 

\subsection{Determine between the presence of a CRW and a BRW.}

First, we consider determining between trajectories which are either formed by a biased random walk (BRW) or CRW.  Recall, a CRW can be defined by the step-length and turning angle, $\theta$, whereas in comparison, a BRW where the bias is assumed to be towards a constant direction (point at infinity), can be described by the step-length and \textit{orientation}, $\phi$, which is the angle taken relative to some absolute direction; usually either a bearing from North, or from the positive $x$-axis \citep{codling2008random,bailey2018navigational}.  In this case we would have $\phi$ being described by a circular distribution, $g_{\circ}$, parameterised by, $\mu$, the direction of the bias in relation to the absolute direction of measure and $\rho_{\text{BRW}}$, the length of the mean resultant vector in the direction of $\mu$.  Similar to a CRW, the MRV, $\rho_\text{BRW}$ determines the `strength' of the bias with values closer to 1 giving more straight-line movement, and values closer to 0 giving more meandering movement. Despite not being characterised by the turning angles, a BRW formulated in this way still maintains a turning angle distribution described by some zero centred circular distribution.  Therefore identifying whether movement is best described by a BRW or a CRW cannot simply be determined by considering the turning angle distribution.

There are various methods to determine whether recorded movement is best described by a BRW or a CRW \citep{benhamou2006detecting,marsh1988form}.  However, these methods are not always reliable nor practical, often resulting in large variances in the measured statistic, the need to compare the measured statistic from observed data with that of large scale simulations, or an over reliance on best fitting distributions to data sets, as well as requiring a trajectory featuring a large number of steps \citet{benhamou2006detecting}.  Here we consider short paths and demonstrate that by simply comparing the observed distribution of turning angles when a movement path is sub-sampled, with the expected distribution as predicted by Theorem 2, gives a simple qualitative analysis of whether a CRW or BRW is the best descriptive model.

\subsubsection{Simulated Data}
Trajectories of length 100 steps were simulated, formed either using a BRW or CRW model.  In each case, step-lengths were fixed at unit length, with the turning angles described by a zero-centred wrapped normal distribution with $\rho_{\text{CRW}}=0.9, 0.7, 0.5$ for the CRW, and orientations, $\phi$ described by a zero-centred wrapped normal distribution with $\rho_{\text{BRW}}=0.95, 0.84, 0.71$.  These gave the BRW a turning angle distribution equivalent to the values chosen for $\rho_{\text{CRW}}$, therefore, each trajectory, whether BRW or CRW, was expected to have the same step length and turning angle distributions.

Fig 10 shows the turning angle distribution of un-sampled trajectories (solid black) for each value of $\rho$, formed by either a BRW or CRW, along with the expected sub-sampled turning angle distribution for a CRW as given by Theorem 1 (black dashed).  In each panel, the resulting distributions of turning angles found when the trajectories are sub-sampled at a rate of $r=2$ are shown in: light red, for those that were BRW; and light blue for those that were CRW.  They demonstrate that if the sub-sampled distribution of turning angles has a sharper peak than the original distribution, the movement is best described by a BRW, whereas a CRW is expected to have a flatter peak, below the un-sampled turning angle distribution.  The clear delineation between the results from the BRW (red) and the CRW (blue) indicate that this method can be used on short movement paths to provide an indication of whether movement is likely from a BRW or CRW.
 
Whilst these results do not give a quantitative result for the expected likelihood of movement being either a BRW or CRW, the results do provide a clear qualitative indication that, crucially, does not rely on large scale repeated simulation or best-fitting distributions to data.

\begin{figure}
    \centering
    \includegraphics[width=\linewidth]{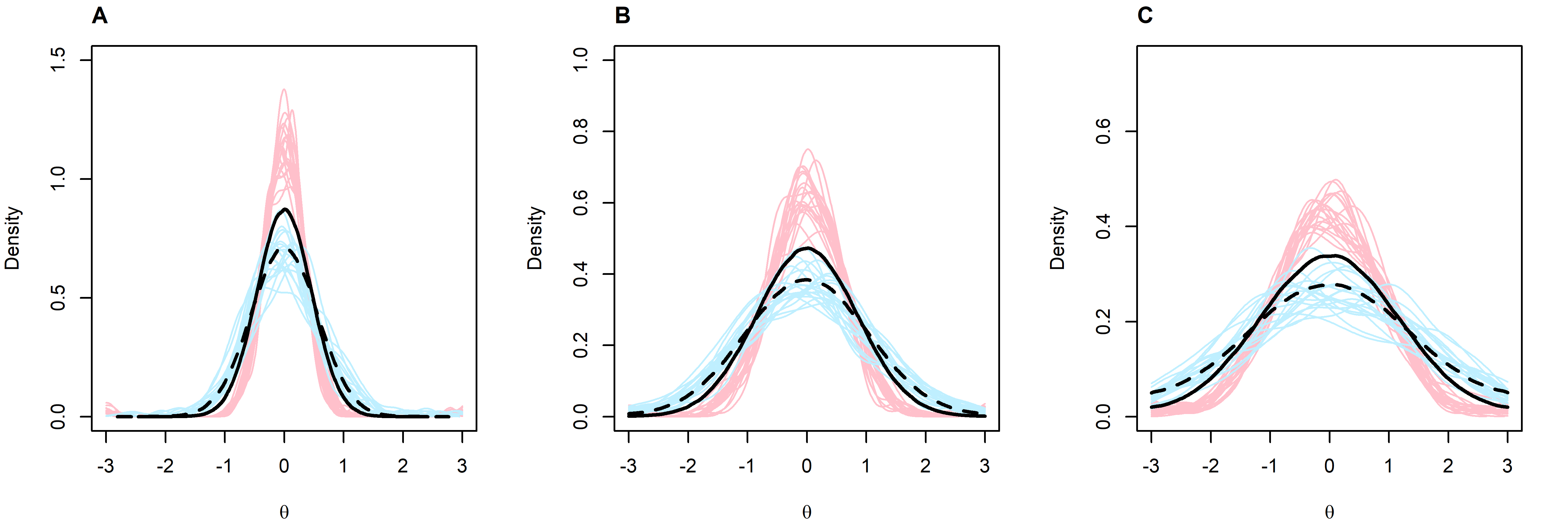}
    \caption{Fig 10 depicts the results of the approach outlined in section 4.1.1 repeated 20 times for both a BRW (pale red) and a CRW (pale blue). }
    \label{fig:app B v CRW}
\end{figure}

\subsubsection{Observed data - \textit{P. cupreus}}
We now demonstrate this approach on observed experimental movement data.  In \citet{bailey2021walking} \textit{P. cupreus} beetles were tracked with their resulting movement classified as either a BRW, CRW or undetermined using the Marsh-Jones statistic \citep{marsh1988form}.  That approach requires multiple repeated simulation to produce a statistic to give a quantitative indication as to which RW model best describes the data, however, the resulting statistic can produce wide confidence intervals, which was compounded  in this instance by the relative small data set, with each trajectory consisting of 300 time steps.  We can compare the results found in \citet{bailey2021walking} using the Marsh-Jones statistic with the method discussed above.  Fig 11 shows the recorded movement paths (A) for four beetles, along with the distribution of the turning angles (B - solid black line).  In the original manuscript, the movement from Beetle 5 and Beetle 9 were classified as a CRW, Beetle 4 a BRW and Beetle 1 was undetermined (see Supplementary Material File 1 in \citet{bailey2021walking}).  If we consider the sub-sampled turning angle distributions (dashed-dotted red line) and compare these with the expected result for a pure CRW (dashed blue line), we can see that Beetle 5 and Beetle 9 seem closely related to the predicted results indicating a CRW is potentially an appropriate model.  Whereas, Beetle 4 has a much sharper peak than the initial distribution and therefore would be best described by a BRW. Beetle 1 has a sub-sampled distribution very similar to the initial distribution and therefore can not be reliable described by either a BRW or CRW.  These results therefore match those found in the initial analysis, but  using a simpler approach.
\begin{figure}
    \centering
    \includegraphics[width=6in]{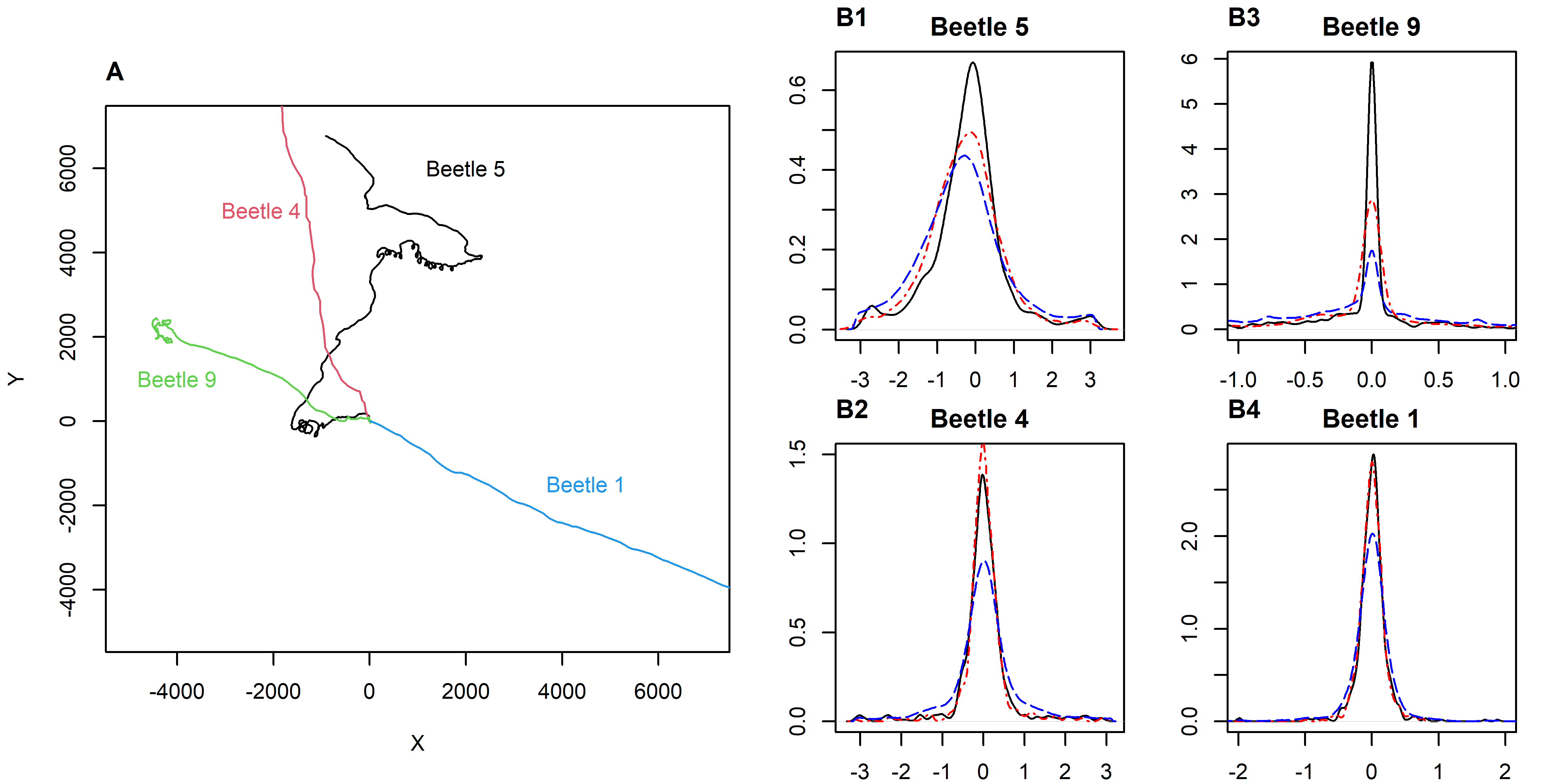}
    \caption{(A) depicts the movement trajectories of the four \textit{P. cupreus} beetles as described in Section 3.1.2 and by \citet{bailey2021walking}.  (B1-4) shows the turning angle distribution of the un-sampled paths (black, solid line), the turning angel distribution found when sub-sampling the trajectories as a rate of $r=2$ (red, dashed-dotted line), against the expected sub-sampled turning angle distribution as calculated using Eq. 9 (blue, dashed line).}
    \label{fig:app beetle}
\end{figure}
\subsection{Identification of multi-state or more complex movement}
Recorded movement often covers multiple behaviours in individuals.  For example, animals can exhibit distinct movement phases across a period of time such as foraging, which can be characterised by slow winding movement, and travelling, where movement is faster and more directed \citep{gurarie2016animal}.  These can be identified using various methods, including Hidden Markov models \citep{michelot2016movehmm}, break point analyses \citep{edelhoff2016path} and mixed distributions \citep{bailey2021emergence}.  However, as discussed in \citet{bailey2021emergence} it is not always clear when such a multi-state model should be considered by examining the data.  Here we show how by considering the sub-sampled turning angle distribution of observed movement data, an indication that a multi-state model would be best suited to the data rather than a simple single-state RW model, can be inferred.  This approach does not give specific information regarding the multi-state model, such as the number of states, the shape of the distributions in each state or the time spent in each state; rather it can be used as a first pass to help indicate whether a more in-depth analysis to determine the best fitting multi-state model should be used.  With the approach requiring a relatively small number of data points to give a clear indication.

\subsubsection{Simulated Data}
Fig \ref{fig:app multi}A shows an example of a movement path with two distinct movement states: state 1, characterised by straighter movement (corresponding to a turning angle distribution with a high mean resultant vector, $\rho$); and state 2, featuring more winding movement (corresponding to a turning angle distribution with a low mean resultant vector).  The corresponding turning angle distribution for the entire movement path is shown in Fig \ref{fig:app multi}B (black lines).  Using this distribution, the expected distribution for turning angles from the sub-sampled path can be found using Eq. 9 (dashed blue line), this can then be compared to the distribution of turning angles from the sub-sampled movement path (red lines).  Fig \ref{fig:app multi}B shows the results from repeatedly simulating a mixed movement path of total length 300 steps, formed by repeated periods of 50 steps in state 1 (straighter movement) followed by 50 steps in state 2 (winding movement).  The lighter lines in Fig \ref{fig:app multi}B show the results from individual simulations, with the thick lines showing the averaged results from the repeated simulations.  These results indicate that when movement is formed from a multi-state CRW the sub-sampled turning angle distribution (red solid) features a peak laying below the original un-sampled turning angle distribution (black solid) but above the expected sub-sampled distribution for a pure CRW (blue dashed).

\subsubsection{Observed data - African Elephant}
As an example with real-world data we consider the movement path of an bull African elephant, \textit{Loxodonta africana}, (Fig \ref{fig:app multi}C) which was demonstrated by  \citet{bailey2021emergence} to be best described by a multi-state CRW (data available from \citet{wall2014elliptical}).  However, as discussed by \citet{bailey2021emergence}, the turning angle distribution of the entire movement path is well described by a single distribution and could therefore be misclassified as a single movement state.    Whereas \citet{bailey2021emergence} demonstrated the better fitting multi-state CRW model by performing a parameter sweep requiring large scale simulations, here we indicate the presence of such multi-sate behaviour by subsampling.  Fig \ref{fig:app multi}D shows the distribution of the turning angles from the observed data (black) along with the distribution of the sub-sampled turning angles (red) and the expected distribution of the sub-sampled turning angles when using Theorem 2 (blue).  The results indicate that the movement is not well described by a pure CRW as the blue and red lines are not closely matched, however as the initial and sub-sampled distributions are closely matched (red and black) this would indicate that a more complex movement model should be explored.

\begin{figure}
    \centering
    \includegraphics[width=4in]{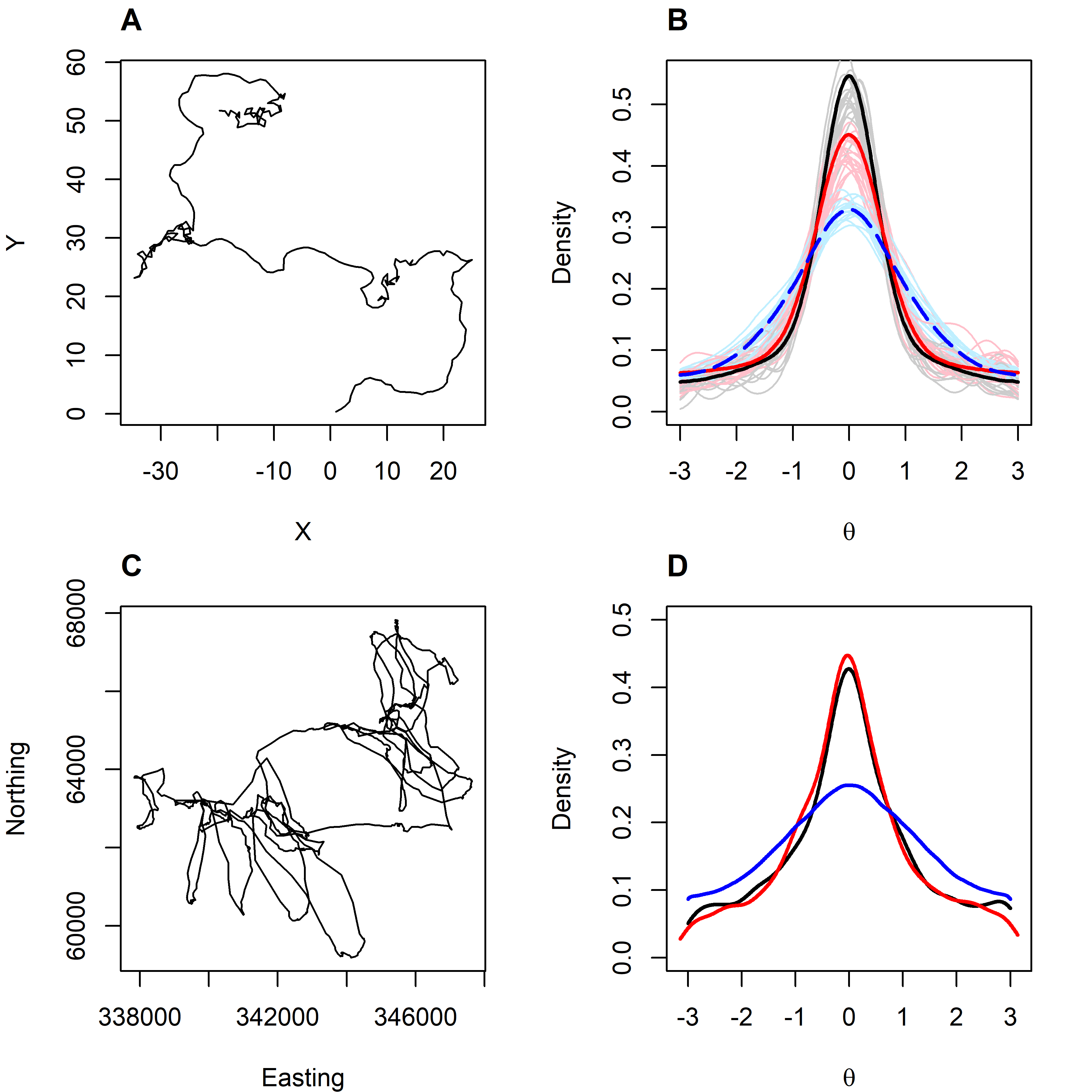}
    \caption{(A) Depicts an example of a multi-state movement path featuring two distinct movement phases.  (B) Plots the turning angle distributions for the multi-state movement depicted in (A); the distribution from the upsampled movement path (black), the distribution from the sub-sampled movement (red) and the expected distribution for the sub-sampled path assuming movement was formed from a single CRW (blue).  The lighter lines correspond to individual simulations, the bold lines correspond the the expected results. (C) Plots the movement path of the African elephant (ID: Haliba) from \citet{wall2014elliptical}.  (D) plots the distribution of the turning angles from (C) from the un-sampled movement (black), the sub-sampled movement (red) sampled at rate $r=2$, and the expected sub-sampled distribution if the initial movement was from a single CRW (blue)}
    \label{fig:app multi}
\end{figure}

\section{Discussion}

Here we have investigated the effect of sub-sampling on the descriptive probability distributions of correlated random walks.  We have given precise formulations for the step-length and turning angle distributions when a CRW is sub-sampled at a rate of every other location ($r=2$) assuming the initial step sizes were all constant.  We have further shown it is possible to determine (up to a very close approximation) the form of the PDF for the turning angles of a sub-sampled CRW for any sub-sampling rate of the form $r=2^k$. We have demonstrated the usefulness of sub-sampling as a form of data analysis by highlighting approaches where it can help in identifying the most appropriate model to describe movement in several scenarios, with our findings having applications in various fields across the sciences, specifically in the field of movement ecology.

There are many ways this work could be advanced such as by obtaining expressions for sampling at any rate $r$.  Our presented method works only for powers of 2, and cannot be extended to allow for all integer sampling rates, therefore some alternative method that does not rely on triangulation may be required.  We also did not consider extending this process to CRWs in 3-dimensions.  This is an active area of research, due to the increasing accuracy and availability of telemetry data recorded in 3-dimensions, this has lead to many recent advances in developing the mathematical properties of 3-dimensional CRW \citep{ahmed2023random,sadjadi2015persistent}.

Similarly, there is more work required to properly include the effect of step-length distribution.  Whilst here the effect of differing initial step-lengths were not considered as they were found to have a negligible effect on the result compared to the sample rate, turning angle distribution and mean-resultant vector of the turning angle distribution (see Remark 2 and Appendix C), it is clear that there is an effect.  More importantly, there will be auto-correlation between step-lengths and turning angles that are not considered here at higher sampling rates, rather, all sub-sampled paths were still assumed to consist of turning angles and step-lengths that are i.i.d as well as neither auto- nor cross-correlated.  This has been explored when considering uniform turning angle distributions in \citet{rosser2013effect}, but including more general turning angle distributions and sampling at any level is yet to be solved.

Another natural extension would be to consider the effects of sampling on other, more complex, RW models, as briefly explored in the Applications (Section 4). When considering animal movement, individuals will often have both a preferred overall direction of movement (bias) and a preference in following their previous direction (persistence), this gives rise to the biased and correlated random walk (BCRW) and is usually characterised by some balancing combination of a BRW and a CRW.  These can be created in a variety of approaches, such as vector weighted \citep{benhamou1992efficiency,bailey2018navigational}, weighted angular \citep{schultz2001edge} or through a switching process \citep{peleg2016optimal}.  Similarly, there are many other approaches used to describe motion in general, from across the sciences, in both discrete and continuous time, such as velocity-jump \citep{monmarche2020velocity,othmer2000diffusion}, Ornstein-Uhlenbeck \citep{sevilla2019generalized,breed2017predicting}, continuous-time correlated random walks \citep{johnson2008continuous}, (truncated) L\'{e}vy walks \citep{humphries2013new,zaburdaev2015levy}, persistent turner walks \citep{gautrais2009analyzing}, along with more statistical approaches such as step-selection functions \citep{thurfjell2014applications,duchesne2015equivalence}, Brownian bridges \citep{horne2007analyzing} and many others.  All these have different formulations which inform and describe the subsequent movement path, though the connections between models are being explored \cite{duchesne2015equivalence}.  Therefore the effect sampling has on these underlying parameters and models has yet to be fully explored outside of the discrete time RW model.  Specifically, the effect of transforming a continuous movement path, to discrete paths is known to have significant effects on the analysis of the movement, therefore understanding more precisely how the frequency of this discretisation and the links between different discretisations would help retain the important information of the movement path \citep{mcclintock2014discrete}. 

We also only considered that the initial CRW was described by one turning angle and one step-length distribution, whereas as discussed in the Applications section, this is usually not the case in movement data, with movement being formed from the combination of multiple movement models, each with their own parametrisation.  
Whilst we showed this method of considering the sub-sampled distributions can give an indication as to whether multiple states of movement are present, as well as whether movement is best described by a CRW or a BRW this is only a qualitative analysis.  Further work could look at using this methodology to provide a metric or quantitative measure as to the likely model between telemetry data, perhaps similar to the Marsh-Jones statistic \citep{marsh1988form}.

Whilst we have extended the understanding of the effects of sub-sampling, our results in Theorems 1 \& 2 lead to the following conjectures about the precise connection between the distributions of turning angles from sub-sampled CRW trajectories
\begin{conjecture}
    There is no circular distribution $f_\circ (\theta;\lambda_1,\ldots,\lambda_n)$, other than the circular uniform distribution, such that $f^{\{2\}}_{\circ}=f_\circ (\theta;\hat{\lambda}_1,\ldots,\hat{\lambda}_n)$, regardless of step-length distribution $\Lambda$.
\end{conjecture}
or more generally
\begin{conjecture}
    There is no circular distribution $f_\circ (\theta;\lambda_1,\ldots,\lambda_n)$, other than the circular uniform distribution, such that $f^{\{r\}}_{\circ}=f_\circ (\theta;\hat{\lambda}_1,\ldots,\hat{\lambda}_n)$, where $r\geq2$ an integer, regardless of step-length distribution $\Lambda$. 
\end{conjecture}
Conjecture 2 is not one that this work can necessarily answer as we only provide an expression for $f_\circ^{\{r\}}$ up to close approximation.  However, Conjecture 1 would follow as a direct consequence of Theorem 1 and it is therefore likely that no subsequent higher sampling rate would work.  

\begin{figure}[h]
    \centering
    \includegraphics[width=2.5in]{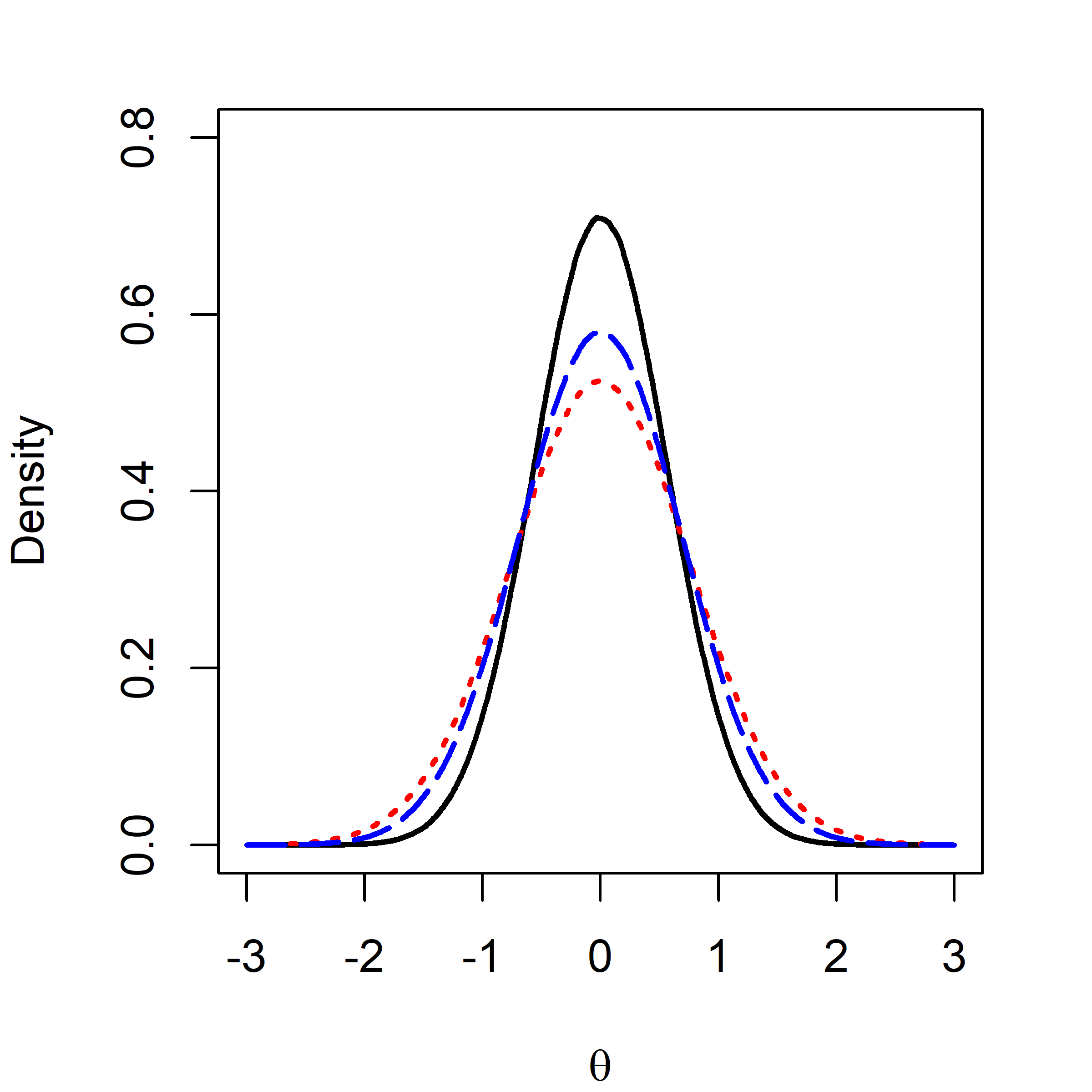}
    \caption{Plot showing the distribution of turning angles when a CRW is sampled at a rate of $r=2$ (black - solid), along with the distribution of turning angles when the same CRW is sampled at rate $r=4$ (red - dotted), which is compared with the distribution found when a CRW with turning angle distribution given by the solid black line is sampled at a rate of $r=2$ (blue - dashed), along with the corresponding step-length distribution given by the CRW corresponding to the solid black distribution.}
    \label{fig:discussion}
\end{figure}
One interesting outcome from our work is that it demonstrates the following.  Let $(X)_i$ be a CRW with turning angle distribution $f_{\circ,X}$ and step-length distribution $\Lambda_X$, then from Theorem 1 we know that the subsequent RW formed from sampling $(X)_i$ at a rate of $r=2$ has turning angle distribution $f^{\{2\}}_{\circ,X}$ and step-length distribution $\Lambda^{\{2\}}_X$.  Now let $(Y)_j$ be a CRW with turning angle distribution $f_{\circ,Y}=f^{\{2\}}_{\circ,X}$ and step-length distribution $\Lambda_Y=\Lambda^{\{2\}}$, then the RW found by sampling  $(Y)_j$ at a rate of $r=2$ will not yield the same distribution as that given by subsampling $(X)_i$ at a rate of $r=4$.  That is $f^{\{2\}}_{\circ,Y}\neq f^{\{4\}}_{\circ,X}$ (see Fig \ref{fig:discussion}).  In theory this could be used to determine if a CRW has previously been sampled, although in practice this will have limited uses as data rarely is known to come from a precise CRW process.  However, this could highlight the importance of auto-correlation in the sampled RW, as demonstrated by \citet{nams2013sampling}, the act of sampling a purely CRW imparts auto-correlation between the step-lengths and turning angles into the resulting sampled RW, therefore this could be what drives this difference.  If so, our method demonstrates the mathematical reasons that underpins this observation giving a quantitative measure of the difference the added auto-correlation causes to the turning angle and step-length distributions.  More work could be done to understand the precise mathematical description between sampling and auto-correlation.

Overall, the CRW remains a useful and powerful tool for modelling a variety of natural processes and despite its ubiquity across the sciences there are still many fundamental mathematical and statistical properties that we do not fully understand, including those pertaining to the effects and usefulness of sub-sampling.  Here we have looked to address some of these properties, however, as discussed our work leads to more questions and points to the need for continued research into CRW and RW theory in general, to help us more fully understand, model and predict these natural processes and phenomena.

\section{Statements and Declarations}
No funding was received to assist with the preparation of this manuscript.  The authors have no relevant financial or non-financial interests to disclose.

\section{Ethics declarations}
\subsection{Conflict of interest/Competing interests} The authors declare that there is no conflict of interest in this paper.
\subsection{Consent to participate} Not applicable.
\subsection{Ethics approval} Not applicable.

\section{Data Availability}
The authors declare that the data supporting the findings of this study are available within the paper, and references therein.

\newpage

\begin{appendices}

\section{Circular and Wrapped Distributions}\label{secA1}
Here we give some brief background notes on wrapped and circular distributions.  More information about circular statistics and data can be found in \citet{jammalamadaka2001topics,fisher1995statistical,pewsey2013circular,mardia2009directional} with a helpful short review provided by \citet{lee2010circular}.

\subsection{General Information}
Circular distributions are probability distributions which have domain $(-\pi,\pi]$ and are therefore frequently used in the description of angular data.  A distribution, $f(X)$, can be readily transformed into a circular distribution, $f_\circ(\theta)$ by wrapping it around the unit circle, that is reducing the domain modulo $2\pi$.  Giving
\begin{equation}
    \theta\equiv X\;\mod{2\pi}
\end{equation}
and allowing the subsequent circular distribution to be considered as
\begin{equation}
    f_\circ(\theta)=\sum_{m=-\infty}^{\infty}f(\theta+2\pi m),\quad \theta\in(-\pi,\pi].
    \label{eq: wrap def}
\end{equation}
Distributions which are found by following this wrapping procedure are known as `wrapped distributions' with commonly used examples including the wrapped Cauchy, wrapped normal \citep{bailey2021emergence}, but many others are also used in circular data analysis including wrapped exponential, wrapped beta and wrapped generalised Gompertz distribution.  Distributions that do not have an equivalent `unwrapped' form are referred to as circular distributions and include the von Mises circular distribution, which is known to closely mimic the wrapped normal distribution \citep{stephens1963random,collett1981discriminating}, and is one of the most commonly used to describe angular data.
\subsection{Moments and mean resultant vector (concentration parameter), $\rho$}

An important notion in circular statistics are the $n$th trigonometric moments, $\phi_n$, (which are analogous  standard \textit{moments} in probability theory).  These are given by
\begin{equation}
    \phi_n=\E[e^{in\theta}]=\int_{-\pi}^{\pi}e^{in\theta}f_\circ(\theta)\:d\theta, \;n\in\mathbb{Z}^+
\end{equation}
alternatively, they can be considered in terms of cosine and sine moments,
\begin{align}\nonumber
    \phi_n&=\E[e^{in\theta}]\\\nonumber
    &=\E[\cos{n\theta}+i\sin{n\theta}]\\\nonumber
    &=\int_{-\pi}^{\pi}\cos{(n\theta)}f_\circ(\theta)\;d\theta+i\int_{-\pi}^{\pi}\sin{(n\theta)}f_\circ(\theta)\;d\theta\\\nonumber
    &=:\alpha_n+i\beta_n.
\end{align}
The first trigonometric moment, $\phi_1$, which determines the mean of the distribution, can be written as
\begin{equation}
    \phi_1=\E[e^{i\theta}]=\rho e^{i\mu}
\end{equation}
where $\rho$ is the length of the mean resultant vector (MRV) and $\mu$ the mean direction of the distribution.

In circular statistics, the measure of how `peaked' a distribution is, is given by the length of the mean resultant vector, often also called the \textit{concentration parameter}, which from above we can see is given by
\begin{equation}
    \rho=\sqrt{\alpha_1^2+\beta_1^2}
\end{equation}
where $\alpha_1$ is the mean cosine and $\beta_1$ the mean sine, of the distribution.

\subsection{Example Wrapped and Circular distributions}
Here we define a few of the most commonly used wrapped and circular distributions.\\

\noindent\underline{Circular uniform distribution}

\noindent The circular equivalent to the uniform distribution is given by

\begin{equation}
    f_{\text{U}}(\theta)=\frac{1}{2\pi},\quad \theta\in(-\pi,\pi].
\end{equation}
This equates to all angles being equally likely to be chosen and is the turning angle distribution which corresponds to directionless movement as seen in a simple random walk, Brownian motion and L\`{e}vy walks \citep{codling2008random,ahmed2023random}.\\

\noindent\underline{Symmetric wrapped stable distribution }

\noindent A symmetric wrapped stable (SWS) distribution is a family of wrapped distributions given by the density function:
\begin{equation}
f_{\text{sws}}(\theta;\rho,\mu,a)=\frac{1}{2\pi}\left(1+2\sum_{n=1}^{\infty}\rho^{n^a}\cos{n(\theta-\mu)}\right),\;\;n\in\mathbb{N},\quad \theta\in(-\pi,\pi]
\end{equation}
where $\rho\in[0,1)$ is the MRV/concentration parameter, $\mu\in[-\pi,\pi)$ is the mean location parameter around which the distribution is symmetric and $a\in(0,2]$ \citep{jammalamadaka2001topics,bailey2021emergence}.  SWS distributions are unimodal distributions, which are symmetric around the mean value of $\mu$.\\

\noindent\underline{Wrapped Normal distribution}

\noindent The wrapped normal (WN) distribution is found by taking the standard normal distribution on the reals, and wrapping it around the unit circle as in Eq. \ref{eq: wrap def}.  This can be represented as
\begin{equation}
f_{\text{WN}}(\theta;\rho,\mu)=\frac{1}{2\pi}\left(1+2\sum_{n=1}^{\infty}\rho^{n^2}\cos{n(\theta-\mu)}\right),\;\;n\in\mathbb{N},\quad \theta\in(-\pi,\pi]
\end{equation}
where $\rho\in[0,1)$ and $\mu\in[-\pi,\pi)$ are as described for the SWS distribution.  Note, this is precisely the from of a SWS with $a=2$, which demonstrates that the wrapped normal is a SWS distribution.\\

\noindent\underline{Wrapped Cauchy distribution}

\noindent The wrapped Cauchy (WC) distribution is found by taking the standard Cauchy distribution on the reals, and wrapping it around the unit circle as in Eq. \ref{eq: wrap def}.  Which can be represented as
\begin{equation}
f_{\text{WC}}(\theta;\rho,\mu)=\frac{1}{2\pi}\left(1+2\sum_{n=1}^{\infty}\rho^{n}\cos{n(\theta-\mu)}\right),\;\;n\in\mathbb{N},\quad \theta\in(-\pi,\pi]
\end{equation}
where $\rho\in[0,1)$ and $\mu\in[-\pi,\pi)$ are as described for the SWS distribution.  Note, this is precisely the from of a SWS with $a=1$, which demonstrates that like the WN, the wrapped Cauchy is also an SWS distribution.  However, unlike the WN, the WC can be written in closed form as
\begin{equation}
    f_{\text{WC}}(\theta;\rho,\mu)=\frac{1}{2\pi}\frac{1-\rho^2}{1+\rho^2-2\rho\cos{(\theta-\mu)}},\quad \theta\in(-\pi,\pi].
\end{equation}
In general, compared with the WN, the WC features a more sharply peaked distribution around the mean value $\mu$, which decays quicker away from the mean but keeping `fatter-tails' at the extreme values \citep{bailey2021emergence}.\\

\noindent\underline{von Mises distribution}

\noindent The von Mises (vM) distribution is a circular distribution which does not have an `un-wrapped' equivalent.  It is given by the density function,
\begin{equation}
f_{\text{vM}}=\frac{1}{2\pi I_0(\kappa)}e^{\kappa\cos(\theta-\mu)},\quad \theta\in(-\pi,\pi]
\end{equation}
where $\mu\in[-\pi,\pi)$, $\kappa>0$ and $I_0(\kappa)$ is the modified Bessel function of the first kind with order 0.  Here, $\kappa$ acts similarly to $\rho$ in the SWS distributions, however it is unbounded above.  The distribution is  symmetric around the value of $\mu$ and, whilst is not a member of the SWS family of distributions, it closely mimics the properties of the WN \citep{codling2010d,collett1981discriminating,stephens1963random}.  Due to this, and as it can be written in closed form, it is often used in place of the WN.

\subsection{Triangle Distribution}
Here we give details regarding the form of the triangle distribution used in the manuscript; note other forms of a circular triangle distribution exist e.g. \cite{jammalamadaka2001topics,circular}.\\

The triangle distribution can be considered as a circular distribution in various ways, for example it could be calculated as a wrapped distribution similar to the wrapped Cauchy and wrapped normal where the wrapped triangle distribution is found by taking the triangle distribution on the reals and taking the angles modulo $2\pi$.  Here we take the usual triangle distribution on the reals and  remove any angles with absolute value greater than $\pi$.  The resulting density is then re-scaled to give a PDF as shown in Fig.\ref{fig:triangle dist} and formulated below.\\

Let $f(x)$ be the triangle distribution on the reals that is symmetric around the point $c\in \mathbb{R}$, given by
\begin{equation}\nonumber
    f(x;a,c)=\frac{(a-c)-\lvert c-x\rvert}{(a-c)^2}
\end{equation}
where $a\in \mathbb{R}_{>0}$.   When the distribution is zero centred we have $c=0$, giving
\begin{equation}\nonumber
    f(x;a)=\frac{a-\lvert x\rvert}{a^2}
\end{equation}
Essentially, this gives a distribution with domain  $[-a, a]$.   If we consider this distribution with domain $[-\pi,\pi]$ , to give a `circular' triangle distribution, which we denote by $f_\circ$, then in the case where $a\leq\pi$ we have
\begin{equation}\nonumber
    f_\circ(x;a)=\begin{cases}f(x;a), \quad \text{for }x\in [-a,a]\\
    0,\qquad \text{otherwise.}
    \end{cases}
\end{equation}
 For any $a>\pi$ we simply restrict $x$ to be in $[-\pi,\pi]$ and re-scale.  That is, if $a>\pi$  we have
\begin{equation}\nonumber
f_\circ(x;a)=\frac{f(x;a)}{\left(\int_\pi^a f(\tilde{x})d\tilde{x}+\int_
{-a}^{-\pi} f(\tilde{x})d\tilde{x}\right)} = \frac{f(x;a)}{F(\lvert x\rvert>a)}=\frac{f(x;a)}{2F(x>a)}
\end{equation} 
for all $x\in[-\pi,\pi]$.  This still holds for the case when $a<\pi$, hence the circular form of the triangle distribution is therefore given by
\begin{equation}\nonumber
    f_\circ(x;a)=\begin{cases}\frac{f(x;a)}{2F(x>a)}, \quad \text{for }x\in [-a,a]\\
    0,\qquad \text{otherwise.}
    \end{cases}
\end{equation} where $f(x;a)=\frac{a-\lvert x\rvert}{a^2}$.\\

The mean resultant vector for these distributions for various values of $a$ are given in Table \ref{table: tri} following the procedure outlined in Chapter 1 of \citet{jammalamadaka2001topics}.

\begin{table}[]
\centering
\begin{tabular}{|c|c|}
\hline
\textbf{$a$}        & \textbf{$|\text{MRV}|=\rho$} \\ \hline
0.3141593           & 0.9918025    \\ \hline
0.6283185           & 0.967536     \\ \hline
0.9424778  & 0.9280932    \\ \hline
1.2566371  & 0.8750400    \\ \hline
1.5707963           & 0.8105887    \\ \hline
1.8849556           & 0.7370167    \\ \hline
2.1991149           & 0.6567193    \\ \hline
2.5132741  & 0.5730016    \\ \hline
2.8274334  & 0.4881268    \\ \hline
3.1415927 & 0.4054206    \\ \hline
3.4557519 & 0.3376897    \\ \hline
3.7699112 & 0.2896437    \\ \hline
4.0840704 & 0.2531472    \\ \hline
4.3982297 & 0.2252001    \\ \hline
4.7123890 & 0.2028957    \\ \hline
5.0265482 & 0.1841670    \\ \hline
5.3407075 & 0.1691151    \\ \hline
5.6548668 & 0.1560478    \\ \hline
 5.9690260 & 0.1445165    \\ \hline
6.2831853 & 0.1344017    \\ \hline
\end{tabular}
\caption{Table showing the value for $a$ used in the circular triangle distribution and the corresponding mean resultant vector length ($|\text{MRV}|=\rho$)}
\label{table: tri}
\end{table}

\begin{figure}[H]
    \centering
    \includegraphics[width=3.5in]{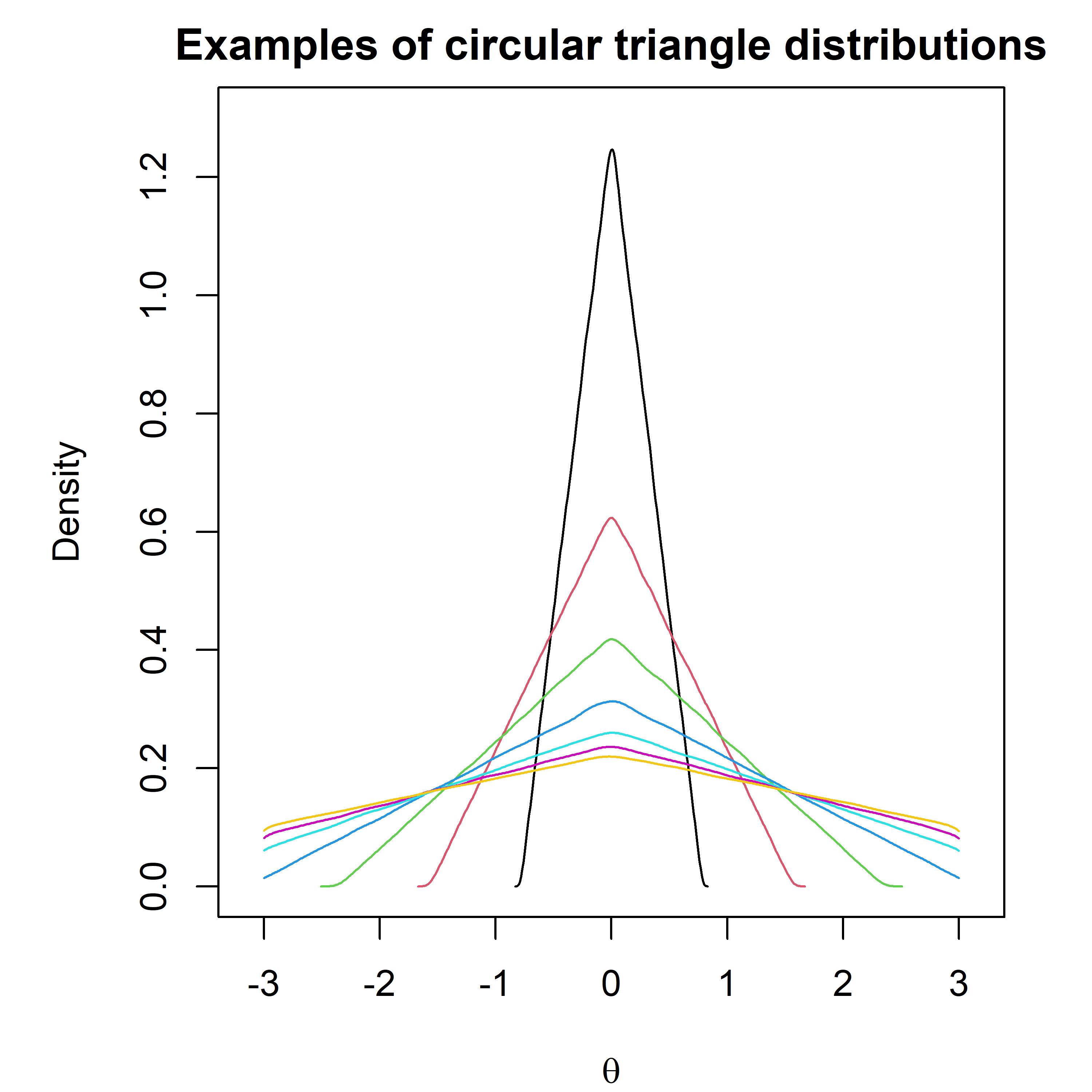}
    \caption{Figure showing examples of the triangle distribution used throughout the text.  Lines correspond to different values of $a$ which is discussed in the text in Appendix D.  Here we have: $a=\frac{\pi}{4}$ - black; $a=\frac{2\pi}{4}$ - red; $a=\frac{3\pi}{4}$ - green; $a=\pi$ - blue; $a=\frac{5\pi}{4}$ - cyan; $a=\frac{6\pi}{4}$ - pink; $a=\frac{7\pi}{4}$ - gold.}
    \label{fig:triangle dist}
\end{figure}

\newpage

\section{Example of Theorem 2, with sample rate of $r=4 \Leftrightarrow k=2$}
As a demonstration we consider the example for sampling a CRW at rate $r=2^2=4$ equivalent to $k=2$.  We first follow Eq. 22 to find the coefficients $a^{\{2^2\}}_{t-\tau}$ and then the approximated coefficients $\tilde{a}^{\{2^2\}}_{t-\tau}$ using Eq 23.\\

From Eqs. 18 \& 19, with $k=2$, we have that at point $X_t$ the turning angle of the sub-sampled random walk is
\begin{equation}
    \theta^{\{2^2\}}_{X_t}=(-1)^{\mathbbm{1}_{1-p_2}}\theta^{\{2^1\}}_{X_t}+(-1)^{\mathbbm{1}_{1-p_2}}\alpha^{\{2^2\}}_{X^-_t}+(-1)^{\mathbbm{1}_{1-p_2}}\alpha^{\{2^2\}}_{X^+_t}
\end{equation}
with
\begin{equation}
    \alpha^{\{2^2\}}_{X^-_t}=\frac{1}{2}\theta^{\{2^1\}}_{X_{t-2^1}},\quad \alpha^{\{2^2\}}_{X^+_t}=\frac{1}{2}\theta^{\{2^1\}}_{X_{t+2^1}}.
\end{equation}
Hence
\begin{equation}
    \theta^{\{2^2\}}_{X_t}=(-1)^{\mathbbm{1}_{1-p_2}}\theta^{\{2^1\}}_{X_t}+(-1)^{\mathbbm{1}_{1-p_2}}\frac{1}{2}\theta^{\{2^1\}}_{X_{t-2}}+(-1)^{\mathbbm{1}_{1-p_2}}\frac{1}{2}\theta^{\{2^1\}}_{X_{t+2}}.
\end{equation}
Now considering $\theta^{\{2^1\}}_{X_t}$ we have
\begin{equation}
    \theta^{\{2^1\}}_{X_{t}}=(-1)^{\mathbbm{1}_{1-p_1}}\theta^{\{2^0\}}_{X_{t}}+(-1)^{\mathbbm{1}_{1-p_1}}\alpha^{\{2^1\}}_{X^-_{t}}+(-1)^{\mathbbm{1}_{1-p_1}}\alpha^{\{2^1\}}_{X^+_{t}}
\end{equation}
with
\begin{align}\nonumber
    \alpha^{\{2^1\}}_{X^-_{t}}&=\frac{1}{2}\theta^{\{2^0\}}_{X_{t-2^0}}\\\nonumber
    &=\frac{1}{2}\theta^{\{1\}}_{X_{t-1}},\\\nonumber
    \text{and}\quad \alpha^{\{2^1\}}_{X^+_{t}}&=\frac{1}{2}\theta^{\{2^0\}}_{X_{t+2^0}}\\\nonumber
    &=\frac{1}{2}\theta^{\{1\}}_{X_{t+1}}.
\end{align}
Using the above and the notation convention outlined in Section 3 that $\theta^{\{1\}}_i\stackrel{\text{def}}{=}\theta_i$, gives
\begin{equation}
    \theta^{\{2^1\}}_{X_{t}}=(-1)^{\mathbbm{1}_{1-p_1}}\theta_{X_{t}}+(-1)^{\mathbbm{1}_{1-p_1}}\frac{1}{2}\theta_{X_{t-1}}+(-1)^{\mathbbm{1}_{1-p_1}}\frac{1}{2}\theta_{X_{t+1}}.
\end{equation}
Repeating for the $\frac{1}{2}\theta^{\{2^1\}}_{X_{t-2}}$ term in Eq B.3 gives
\begin{align}\nonumber
    \theta^{\{2^1\}}_{X_{t-2}}&=(-1)^{\mathbbm{1}_{1-p_1}}\theta^{\{2^0\}}_{X_{t-2}}+(-1)^{\mathbbm{1}_{1-p_1}}\alpha^{\{2^1\}}_{X^-_{t-2}}+(-1)^{\mathbbm{1}_{1-p_1}}\alpha^{\{2^1\}}_{X^+_{t-2}}\\\nonumber
    \text{where}\quad\alpha^{\{2^1\}}_{X^-_{t-2}}&=\frac{1}{2}\theta^{\{2^0\}}_{X_{t-2-2^0}}\\\nonumber
    \Rightarrow&=\frac{1}{2}\theta^{\{2^0\}}_{X_{t-3}}\stackrel{\text{def}}{=}\frac{1}{2}\theta_{X_{t-3}},\\\nonumber
    \text{and} \quad\alpha^{\{2^1\}}_{X^+_{t-2}}&=\frac{1}{2}\theta^{\{2^0\}}_{X_{t-2+2^0}}\\\nonumber
    \Rightarrow&=\frac{1}{2}\theta^{\{2^0\}}_{X_{t-1}}\stackrel{\text{def}}{=}\frac{1}{2}\theta_{X_{t-1}}.
\end{align}
Hence
\begin{equation}
    \theta^{\{2^1\}}_{X_{t-2}}=(-1)^{\mathbbm{1}_{1-p_1}}\theta_{X_{t-2}}+(-1)^{\mathbbm{1}_{1-p_1}}\frac{1}{2}\theta_{X_{t-1}}+(-1)^{\mathbbm{1}_{1-p_1}}\frac{1}{2}\theta_{X_{t-3}}.
\end{equation}
Similarly, we have for the $\frac{1}{2}\theta^{\{2^1\}}_{X_{t+2}}$ term in Eq. B.3 that
\begin{equation}
    \theta^{\{2^1\}}_{X_{t+2}}=(-1)^{\mathbbm{1}_{1-p_1}}\theta_{X_{t+2}}+(-1)^{\mathbbm{1}_{1-p_1}}\frac{1}{2}\theta_{X_{t+1}}+(-1)^{\mathbbm{1}_{1-p_1}}\frac{1}{2}\theta_{X_{t+3}}.\\
\end{equation}
Finally, substituting B.6 and B.7, into B.5 gives
\begin{align}\nonumber
\theta^{\{2^2\}}_{X_t}=&(-1)^{\mathbbm{1}_{1-p_1}+\mathbbm{1}_{1-p_2}}\theta_{X_t}+(-1)^\mathbbm{{1}_{1-p_2}}\frac{1}{2}\left[(-1)^\mathbbm{{1}_{1-p_1}}\theta_{X_{t-2}}+(-1)^\mathbbm{{1}_{1-p_1}}\frac{1}{2}\theta_{X_{t-1}}+(-1)^\mathbbm{{1}_{1-p_1}}\frac{1}{2}\theta_{X_{t-3}}\right]\\
&+(-1)^\mathbbm{{1}_{1-p_2}}\frac{1}{2}\left[(-1)^\mathbbm{{1}_{1-p_1}}\theta_{X_{t+2}}+(-1)^\mathbbm{{1}_{1-p_1}}\frac{1}{2}\theta_{X_{t+1}}+(-1)^\mathbbm{{1}_{1-p_1}}\frac{1}{2}\theta_{X_{t+3}}\right]
\end{align}

\begin{figure}[h]
    \centering
    \includegraphics[width=0.85\linewidth]{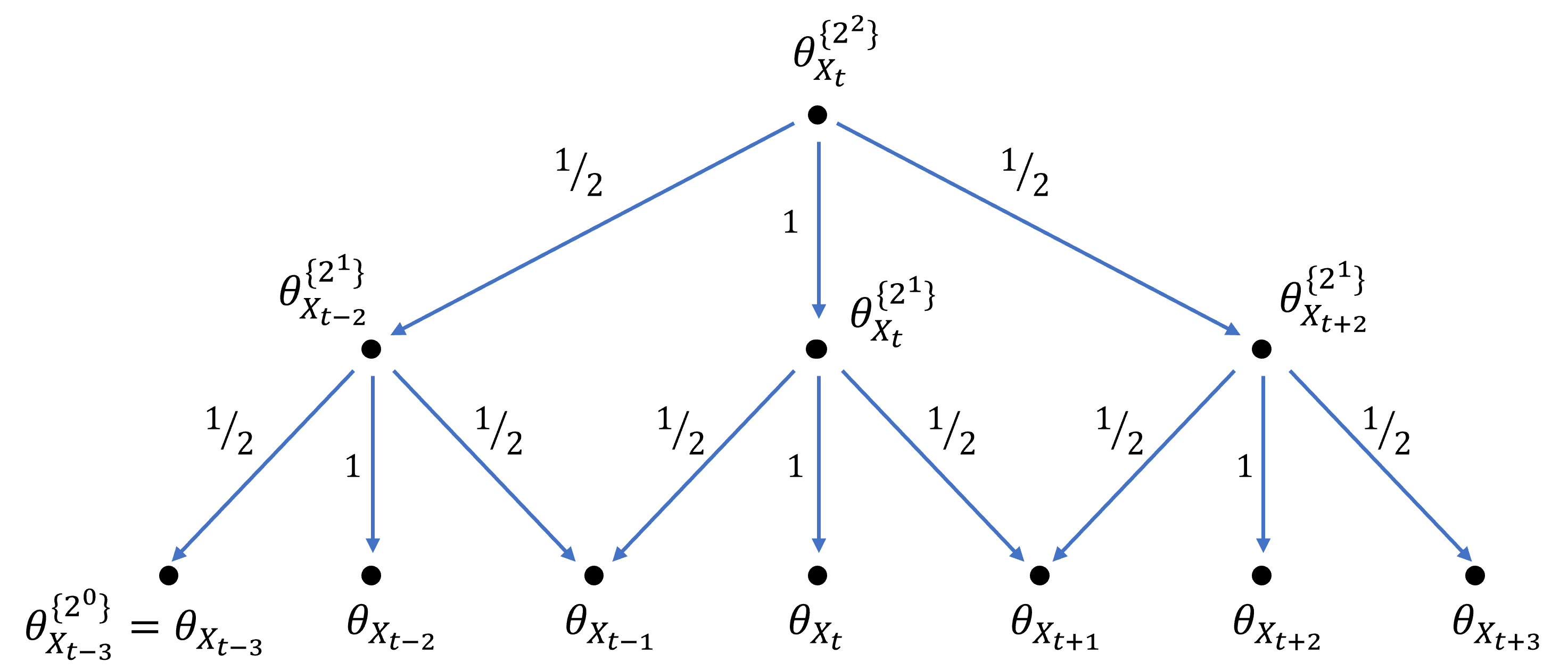}
    \caption{Graphical representation of the calculation for $\theta^{\{2^2\}}_{X_t}$, similar to Figure 6 from the Main text.  The diagram  demonstrates two iterations of Eq. 8 and shows that $\theta^{\{2^2\}}_{X_t}$ can be written as the linear sum of terms $a^{\{2^2\}}_{t-\tau}\theta_{X_{t-\tau}}$ with $\tau=\{-3,-2,\ldots,3\}$, where the coefficients $a^{\{2^2\}}_{t-\tau}$ are determined by considering every path from the top node of $\theta^{\{2^2\}}_{X_t}$ to the leaf node $\theta_{X_{t-\tau}}$ (Eq. A.8).}
    \label{fig:app graph samp 4}
\end{figure}

This demonstrates that to write $\theta^{\{2^2\}}_{X_t}$ as a summation would require considering every possible permutation of the $(-1)^{\mathbbm{1}_{1-p}}$ terms yielding positive or negative terms for both $p_1$ and $p_2$ (see Fig. \ref{fig:app graph samp 4} for graphical representation).

Instead, if we follow the rules in Eq. 23 to give the coefficients $\tilde{a}^{\{2^2\}}_{t-\tau}$ we get the below expression
\begin{equation}
 \theta^{\{2^2\}}_{X_t}=(2p-1)^2\left(\frac{1}{4}\theta_{X_{t-3}}+\frac{2}{4}\theta_{X_{t-2}}+\frac{3}{4}\theta_{X_{t-1}}+\theta_{X_t}+\frac{3}{4}\theta_{X_{t+1}}+\frac{2}{4}\theta_{X_{t+2}}+\frac{1}{4}\theta_{X_{t+3}}\right)
\end{equation}

This allows the distribution of $\theta^{\{2^2\}}_{X_t}$, denoted by $f^{\{4\}}$, to be calculated by convolution, giving;
\begin{align*}
    f^{\{4\}}&(\theta)=\\
    &\frac{4}{(2p-1)^2}f_\circ\left(\frac{4\theta}{(2p-1)^2}\right)\ast\frac{4}{2(2p-1)^2}f_\circ\left(\frac{4\theta}{2(2p-1)^2}\right)\ast \frac{4}{3(2p-1)^2}f_\circ\left(\frac{4\theta}{3(2p-1)^2}\right)\ast \frac{1}{(2p-1)^2}f_\circ\left(\frac{\theta}{(2p-1)^2}\right) \\
    &\ast \frac{4}{3(2p-1)^2}f_\circ\left(\frac{4\theta}{3(2p-1)^2}\right)\ast\frac{4}{2(2p-1)^2}f_\circ\left(\frac{4\theta}{2(2p-1)^2}\right)\ast \frac{4}{(2p-1)^2}f_\circ\left(\frac{4\theta}{(2p-1)^2}\right)\\
    =&\frac{4^6}{(3\cdot2\cdot1)^2}\frac{1}{\left((2p-1)^2\right)^7} f_\circ\left(\frac{\theta}{(2p-1)^2}\right)\ast f_\circ^{\ast2}\left(\frac{4\theta}{(2p-1)^2}\right)\ast f_\circ^{\ast2}\left(\frac{4\theta}{2(2p-1)^2}\right)\ast f_\circ^{\ast2}\left(\frac{4\theta}{3(2p-1)^2}\right)\\
    =&\frac{(p-\frac{1}{2})^{-14}}{(3\cdot2\cdot1)^2} f_\circ\left(\frac{\theta}{(2p-1)^2}\right)\ast f_\circ^{\ast2}\left(\frac{4\theta}{(2p-1)^2}\right)\ast f_\circ^{\ast2}\left(\frac{4\theta}{2(2p-1)^2}\right)\ast f_\circ^{\ast2}\left(\frac{4\theta}{3(2p-1)^2}\right),
\end{align*}
which we note is equivalent to the result given by Theorem 2 with $k=2$.

Finally, transforming this into a distribution on the unit circle gives the final distribution,

\begin{equation}
   f^{\{4\}}_\circ(\theta)=\sum_{k=-\infty}^\infty f^{\{4\}}\left(\theta \pm 2k\pi\right). 
\end{equation}

\newpage

\section{Step-length distribution doesn't affect TA distribution}

Here we show that the choice of distribution for the step lengths ($\Lambda$) will not affect the resulting distributions of the turning angles from sub-sampled random walks, under the assumption that the distribution describing the step lengths is unimodal or uniform.\\

Figures \ref{fig: step length WN} and \ref{fig: step length WC} below demonstrate that various differing initial step-length distributions, including those with large variances, give very similar results in the turning angle distributions of the sub-sampled random walk.  This indicates that the initial step-length does not have a significant effect on the expected turning angle distribution.  Note, all this work concerns distributions with finite variances and so L\`{e}vy walk and power-law behaviour is not considered.

\begin{table}[h]
\centering
\begin{tabular}{|c|c|c|c|}
\hline
\textbf{Distribution}             & \textbf{Mean} & \textbf{Variance} & \textit{line colour}$^\dagger$\\ \hline
Fixed ($\ell=1$) & 1             & 0              & \textit{red}  \\ \hline
Uniform (min=0, max=1000)        & 500.5         & 83,167        &  \textit{black}   \\ \hline
Exp($\lambda$ = 0.5)              & 2             & 4       &   \textit{cyan}        \\ \hline
Exp($\lambda$ = 1)                & 1             & 1        & \textit{green}         \\ \hline
Exp($\lambda$ = 2)                & 0.5           & 0.25     &     \textit{magenta}     \\ \hline
Log-normal($\mu=0, \sigma=1$)     & 1.649         & 4.671     & \textit{blue}        \\ \hline
Log-normal($\mu=0, \sigma=0.5$)   & 1.133         & 0.365      & \textit{gold}       \\ \hline
Log-normal($\mu=0.5, \sigma=1$)   & 2.718         & 12.696     & \textit{silver}       \\ \hline
Log-normal($\mu=1, \sigma=1$)     & 4.482         & 34.513     & \textit{dark grey}       \\ \hline
Log-normal($\mu=1, \sigma=0.5$)   & 3.080         & 2.695     & \textit{rose}        \\ \hline
Log-normal($\mu=2, \sigma=0.5$)   & 8.373         & 19.912     & \textit{light-green}       \\ \hline
\end{tabular}
\caption{Table detailing the step-length distributions used here. Including their respective means and variance. $\dagger$The colours listed here correspond to the colour of the curves used in Fig \ref{fig: step length WN} and \ref{fig: step length WC} below}
\label{Table: step-lengths}
\end{table}

\begin{figure}[H]
    \centerline{\includegraphics{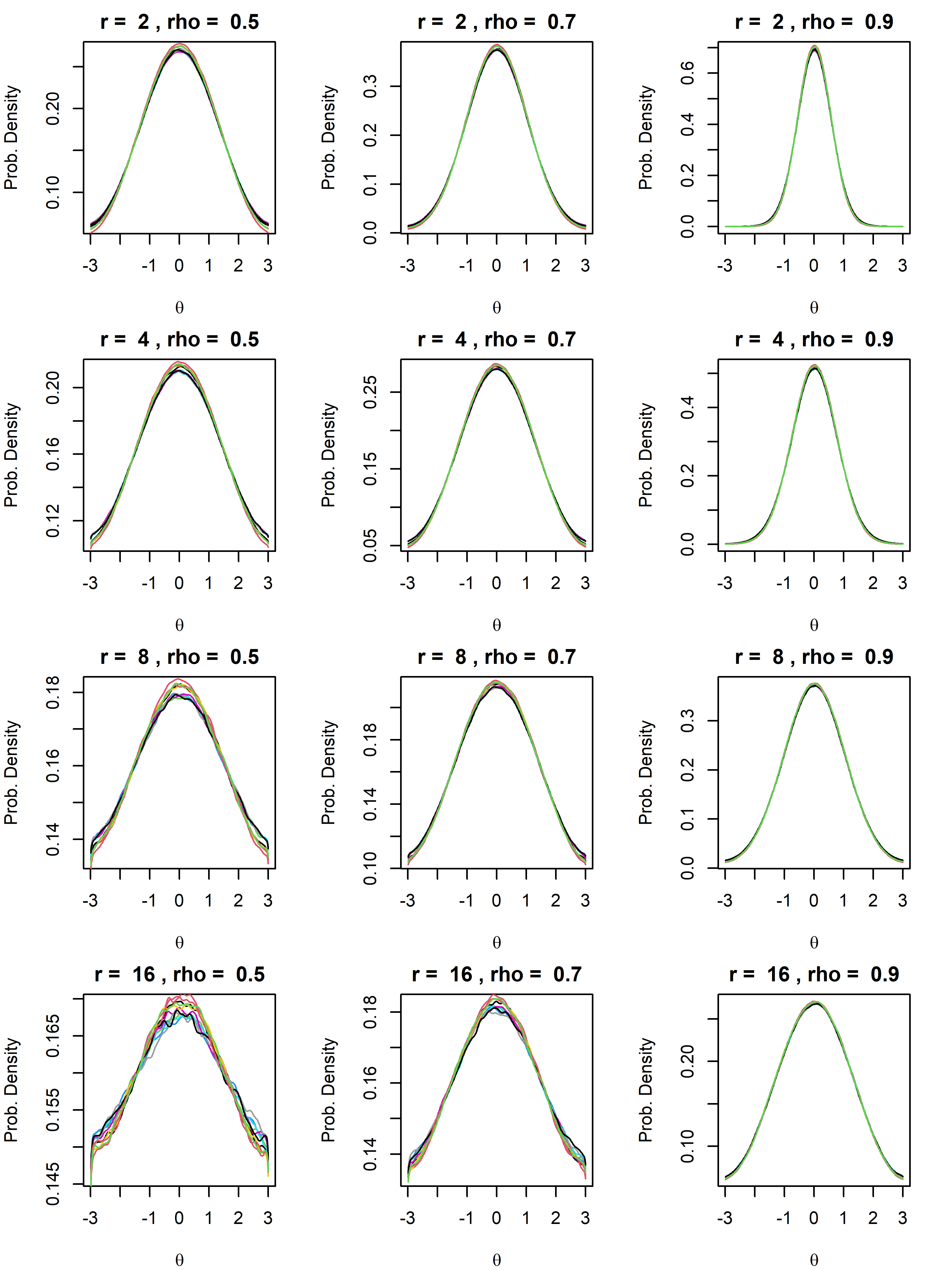}}
    \caption{Figure showing the resulting distribution of turning angles when a RW is sampled at rate $r=2, 4, 8, 16$ (rows), with initial turning angle distribution given by a zero-centred wrapped normal distribution with mean resultant vector, $\rho$, taking values of $\rho=0.5,0.7,0.9$ (columns). In all cases lines correspond to different initial step length distributions as detailed in Table \ref{Table: step-lengths}}
    \label{fig: step length WN}
\end{figure}
\begin{figure}[H]
    \centerline{\includegraphics{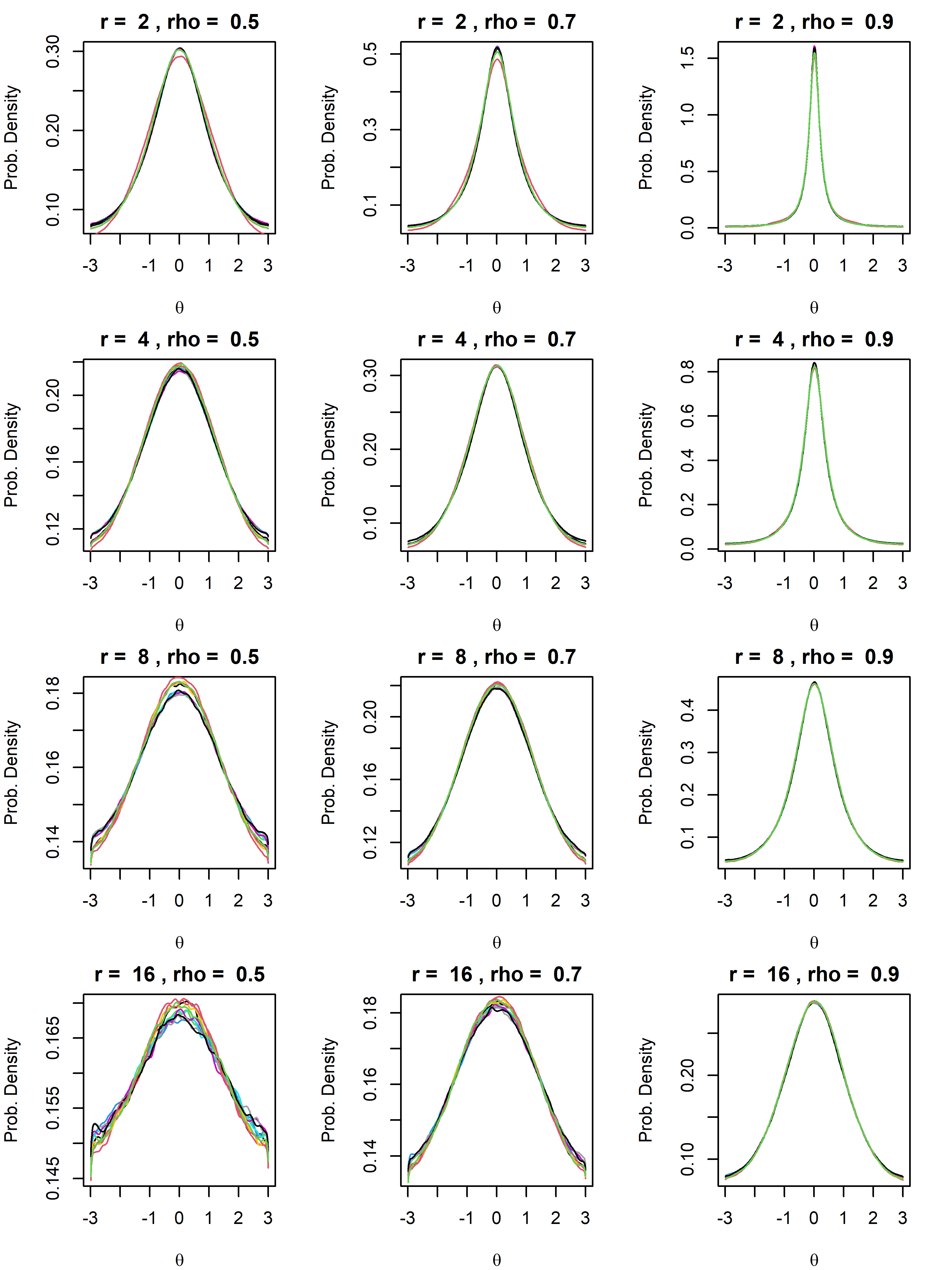}}
    \caption{Figure showing the resulting distribution of turning angles when a RW is sampled at rate $r=2, 4, 8, 16$ (rows), with initial turning angle distribution given by a zero-centred wrapped Cauchy distribution with mean resultant vector, $\rho$, taking values of $\rho=0.5,0.7,0.9$ (columns).  In all cases lines correspond to different initial step length distributions as detailed in \ref{Table: step-lengths}}
    \label{fig: step length WC}
\end{figure}

\end{appendices}

\bibliography{bibliography}

\end{document}